\documentclass[11pt,twoside]{article}
\usepackage{amsmath,amssymb,amsthm}
\usepackage{cite}
\usepackage{graphicx}
\usepackage{subfigure}
\usepackage{color}

\newtheorem{thm}{Theorem}[section]

 \newtheorem{cor}{Corollary}[section]
 \newtheorem{lem}{Lemma}[section]
 \newtheorem{prop}{Proposition}[section]
 \newtheorem{defn}{Definition}[section]
\newtheorem{rem}{Remark}[section]

\def\tilde{\widetilde}
\def\hat{\widehat}

\newcommand\R{\mathbb{R}}

\newcommand\Z{\mathbb{Z}}

\def\cS{{\mathcal S}}

\begin{document}
\title{Global dynamics of Kato's solutions for the 3D incompressible micropolar system}

\author{Zihao Song}
\date{\it \small Research Institute for Mathematical Sciences, Kyoto University\\
Kyoto, 606-8501, Japan\\
 E-mail: szh1995@nuaa.edu.cn}
\maketitle

{\bf Keywords:} Incompressible micropolar system; Global well-posedness; Large time behaviors; critical Besov space\\
{\bf Mathematical Subject Classification 2020}: 76D03, 35Q30, 35D35

\begin{abstract}
We consider the global well-posedness and decay rates for solutions of 3D incompressible micropolar equation in the critical Besov space. Spectrum analysis allows us to find not only parabolic behaviors of solutions, but also damping effect of angular velocity in the low frequencies. Based on this observation, we establish the global well-posedness with more general regularity on the initial data.

The approach concerns large time behaviors is  so-called the Gevrey method which bases on the parabolic mechanics of the micropolar system. This method enables us even to derive decay of \textit{arbitrary} higher order derivatives by transforming it into the control of the radius of analyticity under a particular regularity.
 To bound the growth of the radius of analyticity in general $L^p$ Besov spaces, we shall develop some new techniques concern Gevrey estimates, especially an extended Coifman-Meyer theory, to cover those endpoint Lebesgue framework.
\end{abstract}

\section{Introduction}\setcounter{equation}{0}
In this paper, we consider the following 3-D incompressible micropolar fluid system in $\mathbb{R}^{+}\times\mathbb{R}^{3}:$
\begin{equation}
\left\{
\begin{array}{l}\partial_{t}u-(\chi+\nu)\Delta u+u\cdot\nabla u+\nabla\pi-2\chi\nabla\times\omega=0,\\ [1mm]
 \partial_{t}\omega-\mu\Delta\omega+u\cdot\nabla \omega+4\chi\omega-\kappa\nabla \mathrm{div}\omega-2\chi\nabla\times u =0,\\[1mm]
\mathrm{div}u=0,\\[1mm]
(u,\omega)|_{t=0}=(u_{0},\omega_{0}),
 \end{array} \right.\label{1.1}
\end{equation}
where $u = u(x,t)$ denotes the fluid velocity, $\omega (x,t)$ denotes the velocity field of rotation for the particle of the fluid while $\pi (x,t)$ represents the scalar pressure. $\nu$ denotes the Newtonian kinematic viscosity, $\chi$ denotes the micro-rotation viscosity, and $\kappa,\mu$ are the angular viscosities. In our paper, we assume all those viscosities to be positive.

The system \eqref{1.1} was first introduced by C.A. Eringen to model the micropolar fluids (see \cite{BCD}). The micropolar fluids are fluids with microstructure, which can be viewed as non-Newtonian fluids exhibiting micro-rotational effects and micro-rotational inertia which are significant generalizations of the Navier-Stokes equations covering many more phenomena such as the fluids consisting of the particles suspended in a viscous medium.

Due to its importance in both physics and mathematics, lots of works have been done in the mathematical analysis.
If $\omega=0$ and the micro-rotation viscosity $\chi=0$, the system \eqref{1.1} reduces to the 3-D viscous incompressible Navier-Stokes equations:
 \begin{equation}
\left\{
\begin{array}{l}\partial_{t}u-\nu\Delta u+u\cdot\nabla u+\nabla\pi=0,\\ [1mm]
\mathrm{div}u=0,\\ [1mm]
u|_{t=0}=u_{0}(x).
 \end{array} \right.\label{INS}
\end{equation}
In the fundierend works by Leray \cite{L} and Hopf \cite{H}, authors proved positive results concerned the global-in-time existence of weak solutions (Leray-Hopf solutions) for the incompressible Navier-Stokes equation. However, the uniqueness of Leray-Hopf solutions remains a big open problem. In terms of uniqueness, an important issue is that the Navier-Stokes system owns the following scaling invariance:
\begin{eqnarray}\label{1.2}
(u_{\lambda}(t,x),p_{\lambda}(t,x))=(\lambda u(\lambda^2t,\lambda x),\lambda^2p(\lambda^2t,\lambda x)),\,\,\,\,\,\lambda>0\label{EE}
\end{eqnarray}
let us underline that the critical spaces are the ones that their norms are invariant under the scaling of (\ref{1.2}). Fujita and Kato \cite{FK} proved the global well-posedness results in the critical homogeneous Sobolev space $\dot H^\frac{1}{2}$. Cannone \cite{Ca}, Planchon\cite{P} developed Fujita-Kato's theory in the critical Besov space $\dot B^{\frac{3}{p}-1}_{p,\infty}(\mathbb{R}^3)$, and the existence and uniqueness in $\dot B^{\frac{3}{p}-1}_{p,q}$ with $1\leq p<\infty, 1\leq q\leq\infty$ were established by Chemin in \cite{C}. Bae, Biswas and Tadmor \cite{BBT} proved Gevrey analyticity of solutions in the critical space. Let us emphasize that such critical Besov spaces with negative index allow us to construct solutions for the highly oscillating initial data such like
$$u_{0}(x)=\mathrm{sin}(\frac{x_{3}}{\varepsilon})(-\partial_{2}\Phi(x),\partial_{1}\Phi(x),0),$$
where $\Phi\in\mathcal S(\mathbb{R}^3)$ and $\varepsilon>0$.  As for searching for the largest critical space, one could refer to Koch and Tataru \cite{KT}, Bourgain and Pavlovic \cite{BP}, Wang \cite{Wang} and so on. One may check Lemari$\acute{e}$-Rieusset \cite{L-cras} for complete references in this direction.

For the large time behaviors of solutions for (\ref{INS}), in  \cite{S-ARMA,S-CPDE,S-JAMS}, Schonbek constructed  a complete theory, that is the Fourier splitting method, to establish the asymptotic behaviors for heat equations and weak solutions of Navier-Stokes equations. Precisely, in Schonbek \cite{S-CPDE}, $L^p(1\leq p< 2)$ additional assumption for initial data was introduced. With help of the Fourier splitting method, their results showed once $u_{0}(x)\in L^p\cap L^2(\mathbb{R}^3)$, weak solutions of (\ref{1.1}) fulfill
\begin{eqnarray}\label{heat}\|u(t)\|_{L^2}\lesssim (1+t)^{-\frac{3}{4}(\frac{2}{p}-1)}.\end{eqnarray}
Later, Wiegner \cite{W} furnished a more precise uniform decay rate by carefully analyzing the relationship between
the linear and nonlinear parts of the equation. In Bjorland, Schonbek \cite{BS}, Niche, Schonbek\cite{NS}, the idea of ``decay characterization", which shares a more subtle necessary and sufficient condition for decay rates, was introduced and generalized. Recently, Lorenzo\cite{L} gave a delicate description of  ``decay characterization" in Besov space $\dot{B}^{-2\sigma}_{2,\infty}$, which reflects the characterization of Besov spaces in terms of heat kernel. As for high order derivative estimates, Oliver, Titi\cite{OT} utilized Gevrey analyticity established by Foias, Tenam \cite{FT1,FT2} to present optimal decay rates of high order derivatives for weak solutions in $L^2$.

As for the micropolar system, Galdi, Rionero \cite{GR} and Lukaszewicz \cite{X} proved the existence of the weak solutions. The existence and uniqueness of the strong solutions to the micropolar flows for either local with a large data or global with a small data were considered in \cite{GR} and the references therein. On the well-posedness for the 2D case , one may refer to \cite{TWW,LOA,rotk} respectively; On the blow-up criterion for the smooth solution and the regularity criterion for the weak solution, one refers to \cite{LL,DF} and the references therein. One can also refer to \cite{LZ,catt} for the recent decay results to compressible or incompressible micropolar system.

The major difficulty for (\ref{1.1}) comes from the coupling brought by terms $\nabla\times\omega$ and $\nabla\times u$, which prevents us from the normal energy method together with Fourier localization methods or Kato's semigroup method in \cite{Ca}. To overcome it, Fereirra and Roa in \cite{C} studied the corresponding linearized system:
\begin{equation*}
\left\{
\begin{array}{l}\partial_{t}u-\Delta u-\nabla\times\omega=0,\\ [1mm]
 \partial_{t}\omega-\Delta\omega+4\omega-\nabla \mathrm{div}\omega-\nabla\times u =0,\\[1mm]
 \end{array} \right.
\end{equation*}
where the Green matrix donated by $\mathcal{G}(x,t)$  fulfills the similar pointwise estimate as the heat kernel.
Based on the pointwise estimate, authors got the global existence results in the tempered distributions space $PM^a$-space. When it comes to the critical Besov space in $L^p$ framework,  Chen and Miao \cite{D1} made the best use of divergence free condition and antisymmetric structure of the matric related to the term $\nabla\times u$ and $\nabla\times\omega$. By using the Laplace transform, they got the more explicit expression of $\widehat{\mathcal{G}}(\xi,t)$and corresponding Green matrix estimates where for any couple $(t,\lambda)$ of positive real numbers and $\mathrm{supp}\widehat f\subset \lambda\mathcal{C}$,
\begin{equation*}
\begin{array}{l}\|\mathcal{G}f(x,t)\|_{L^p}\leq Ce^{-c{\lambda}^2t}\|f\|_{L^p} ,  1\leq p\leq\infty.\\ [1mm]
 \end{array}
\end{equation*}
Above estimates for the Green matrix lead to the global well-posedness in the critical Besov space $B^{\frac{3}{p}-1}_{p,\infty}$ with $1\leq p<6$. In \cite{S}, the author extended results of \cite{D1} to
$p\in[1,\infty)$ and establish the Gevrey analyticity in the critical space.

In the present paper, we try to take one step further on well-posedness in \cite{D1,S} first. We shall prove the global existence and uniqueness for Kato's solutions of (\ref{1.1}) with more general initial data concerns angular velocity $\omega$ by utilizing its damping effect and nonlinear structures. Furthermore, we would develop a new Gevrey method, which extends Schonbek's theories to establish decay rates for the solution of any order derivatives but without asking any smallness condition on the initial assumption.

\subsection{Main results}

Before we give our main results, let us first introduce our main functional space. For any $p,q\in[1,\infty]$, the space $E^p_{T}$ is given by
\begin{eqnarray*}
\|(u,\omega)\|_{E^p_{T}}&=&\|u\|^{h}_{{\tilde L}^{\infty}_{T}({{\dot B}^{\frac{3}{p}-1}}_{p,q})\cap{\tilde L}^{1}_{T}({{\dot B}^{\frac{3}{p}+1}}_{p,q})}
+\|\omega\|^{h}_{{\tilde L}^{\infty}_{T}({{\dot B}^{\frac{3}{p}-1}}_{p,q})\cap{\tilde L}^{1}_{T}({{\dot B}^{\frac{3}{p}+1}}_{p,q})}\\
&+&\|u\|^{\ell}_{{\tilde L}^{\infty}_{T}({{\dot B}^{\frac{3}{p}-1}}_{p,q})\cap{\tilde L}^{1}_{T}({{\dot B}^{\frac{3}{p}+1}}_{p,q})}
+\|\omega\|^{\ell}_{{\tilde L}^{\infty}_{T}({{\dot B}^{\frac{3}{p}}}_{p,q})\cap{\tilde L}^{1}_{T}({{\dot B}^{\frac{3}{p}}}_{p,q})}
\end{eqnarray*}
for $T>0$. We also agree with $E^p$ if $T=\infty$.

Now we state our first main result concerns global well-posedness in the critical Besov space:

\begin{thm}\label{thm1}
Let $q\in[1,\infty]$, $p\in[1,6)$. There exist a positive constant $\varepsilon$ such that for $u_{0}\in{{\dot B}^{\frac{3}{p}-1}}_{p,q}$, $\omega_{0}^{h}\in{{\dot B}^{\frac{3}{p}-1}}_{p,q}$ and $\omega_{0}^{\ell}\in{{\dot B}^{\frac{3}{p}}}_{p,q}$ if
\begin{equation*}
\begin{array}{l}
\mathcal{X}_{0,p}\triangleq\|(u_{0},\omega_{0})\|^{h}_{\dot B^{\frac{3}{p}-1}_{p,q}}+\|(u_{0},\Lambda\omega_{0})\|^{\ell}_{\dot B^{\frac{3}{p}-1}_{p,q}}\leq\varepsilon,\\ [1mm]
 \end{array}
\end{equation*}
then the system \eqref{1.1} admits a unique global solution $(u,\omega)\in E^{p^{}}$ satisfying for ant $T>0$
\begin{equation*}
\begin{array}{l}\|(u,\omega)\|_{E^p_{T}}\lesssim \mathcal{X}_{0,p}.\\ [1mm]
 \end{array}
\end{equation*}
\end{thm}
\begin{rem}
Compared to results in \cite{D1,S} where $(u_{0},\omega_{0})\in\dot B^{\frac{3}{p}-1}_{p,q}$, we  relax the restriction of initial data for angular velocity $\omega_{0}$ in low frequencies. The relaxation of regularity for angular velocity dues to finding the damping effect of $\omega$ and the fact the nonlinearities is actually ``linear" in terms of $\omega$.

Careful estimates would be paid while considering nonlinearities, especially those from equations of angular velocity. We state that nonlinear structure and divergence free condition are fully used to cover uniform estimates.
\end{rem}

The second main result is devoted to establishing the large time behaviors of the solution constructed in Theorem \ref{thm1}.

\begin{thm}\label{thm2}
Assume $p\in[1,\infty]$  satisfies the condition in Theorem \ref{thm1} and $(u,\omega)$ be the corresponding global solution constructed in Theorem \ref{thm1}.
We clarify that if  $\sigma$ satisfies $1-\frac{3}{p}<\sigma<\min\{1+\frac{3}{p},1+\frac{3}{p'}\}$ while initial data satisfies
\begin{eqnarray}\label{spaceX}
\mathcal{D}_{0,p}\triangleq\|(u_{0},\Lambda\omega_{0})\|^{\ell}_{\dot B^{-\sigma}_{p,\infty}}<\infty,
\end{eqnarray}
then it holds for all $r\geq p $, $t>0$ such that
\begin{eqnarray}\label{bound decay}\|\Lambda^{l}u\|_{L^{r}}\leq C_{l}  t^{-\frac{\tilde{\sigma}}{2}-\frac{l}{2}}, \quad l>-\tilde{\sigma};\end{eqnarray}
\begin{eqnarray}\label{bound decay 2}\|\Lambda^{l}\omega\|_{L^{r}}\leq C_{l}  t^{-\frac{\tilde{\sigma}}{2}+\frac{1}{2}-\frac{l}{2}} , \quad  l>-\tilde{\sigma}+1\end{eqnarray}
where $C$ is a positive constant depending on $l$ and $\tilde{\sigma}\triangleq\sigma-\frac{3}{r}+\frac{3}{p}$.
\end{thm}

\begin{rem}
Theorem \ref{thm2} presents decay rates of $(u,\omega)$ with arbitrary order derivatives without asking any smallness condition on additional initial assumption.
In fact, one could immediately reach
\begin{eqnarray}\label{bound decay 22}
\|u\|_{\dot B^{\frac{3}{p}-1}_{p,q}}\leq C t^{-\frac{\sigma+\frac{3}{p}-1}{2}},
\end{eqnarray}
which indicates velocity of the solutions constructed in Theorem \ref{thm1} does uniformly decay in the critical space. Decay estimates in (\ref{bound decay 22}) may also be extended to Kato's solutions of the incompressible Navier-Stokes equation which corresponds to $\omega=0$.
\end{rem}

\begin{rem}
One could obtain uniform decay concerns $L^1$ space where
\begin{eqnarray}\|\Lambda^{l}u\|_{L^{1}}\leq C  t^{-\frac{\sigma}{2}-\frac{l}{2}};\quad \|\Lambda^{l}\omega\|_{L^{1}}\leq Ct^{-\frac{\sigma}{2}+\frac{1}{2}-\frac{l}{2}}.\end{eqnarray}
Compared to classical $L^q-L^p$-type estimates for the heat equation or Stokes system, we obtain Gevrey analytical solutions and their uniform decay in this endpoint case.

The main ingredient of obtaining $L^1$ decay is applying an extended Coifman-Meyer theory, which considers Gevrey estimates with Euclidean norm in frequency space. Then the changes of radius of analyticity by constants and equivalence of Gevrey multiplier with $l_{1}$ norm and Euclidean norm allow us to obtain desired estimates.
\end{rem}

\subsection{Strategy}

Let us first give a clear sight for behaviors of solution via the spectrum analysis. Decomposing the angular velocity $\omega=\mathcal{P}\omega+\mathcal{Q}\omega$ into the electromagnetic part $\mathcal{P}\omega=\frac{\nabla\times \nabla\times}{-\Delta}\omega$ and the fluid part $\mathcal{Q}\omega=\frac{\nabla \mathrm{div}}{-\Delta}\omega$ leads to the following extend system:
\begin{equation}\label{3.2}
\left\{
\begin{array}{l}\partial_{t}u-(\chi+\nu)\Delta u-2\chi\nabla\times\mathcal{P}\omega=f,\\ [1mm]
 \partial_{t}\mathcal{P}\omega-\mu\Delta\mathcal{P}\omega+4\chi\mathcal{P}\omega-2\chi\nabla\times u =\mathcal{P} g,\\[1mm]
 \partial_{t}\mathcal{Q}\omega-(\mu+\kappa)\Delta\mathcal{Q}\omega+2\chi\mathcal{Q}\omega=\mathcal{Q} g.\\[1mm]
 \end{array} \right.
\end{equation}

It is clearly that the fluid part $\mathcal{Q}\omega$ shares a parabolic behavior in the high frequencies while behaves as solutions of damping equations in the low frequencies. Regarding the electromagnetic part $\mathcal{P}\omega$, it is convenient to
introduce $\Omega\triangleq \nabla\times \mathcal{P}\omega$. Notice that $\mathrm{div}u=0$, which implies $\nabla\times\nabla\times u=-\Delta u$, the corresponding linear coupling system would be
\begin{equation}\label{3.3}
\left\{
\begin{array}{l}\partial_{t}u-(\chi+\nu)\Delta u-2\chi\Omega=0,\\ [1mm]
 \partial_{t}\Omega-\mu\Delta\Omega+4\chi\Omega+2\chi\Delta u =0.\\[1mm]
 \end{array} \right.
\end{equation}

Consequently, the new variable $(u, \Omega)$ satisfies the coupling $2\times 2$ system:
Taking the Fourier transform with respect to $x\in \mathbb{R}^{3}$ leads to
\begin{equation}\label{yun}
\frac{d}{dt}\left(
              \begin{array}{c}
                \hat{u} \\
                \hat{\Omega} \\
              \end{array}
            \right)
=A(\xi)\left(
              \begin{array}{c}
                \hat{u} \\
                \hat{\Omega} \\
              \end{array}
            \right)
           \quad \mbox{with}\quad A(\xi)=\left(
                                                          \begin{array}{cc}
                                                            (\chi+\nu)|\xi|^2 & -2\chi \\
                                                            -2\chi|\xi|^2 & \mu|\xi|^2+4\chi \\
                                                          \end{array}
                                                        \right).
\end{equation}
A simple calculation shows that the eigenvalues of (\ref{yun}) would be
$$\lambda_{\pm}=\frac{(\bar{\chi}+\mu)|\xi|^2+4\chi\pm\sqrt{\big((\bar{\chi}+\mu)|\xi|^2+4\chi\big)^2-4\bar{\chi}\mu|\xi|^4-16\nu\chi|\xi|^2)}}{2}$$
where $\bar{\chi}=\chi+\nu$.
Apparently eigenvalues indicate that in the low frequency part, i.e. $|\xi|\ll1$, there holds
$$\lambda_{+}\thicksim1;\lambda_{-}\thicksim|\xi|^2$$
which indicates that velocity $u$ behaves as heat equations while angular velocity $\omega$ shares a different behavior in the low frequencies, that is damping effect.
On the other hand, in the high frequencies $|\xi|\gg1$, it is found
$$\lambda_{\pm}\thicksim|\xi|^2 ,$$
where the behavior of  solutions in the high frequencies would be pure parabolic.

Base on the spectrum analysis, we shall treat high frequencies and low frequencies respectively. Our method is close to \cite{S} where the idea of effective velocity would be introduced to decouple the coupling system. On the other hand, more careful nonlinear estimates based on the frequency decomposition would be given and the global existence and uniqueness would be obtained via a standard fixed point argument.

The main idea of establishing large time behaviors of high order derivatives for solutions constructed in Theorem \ref{thm1} comes from the parabolic mechanics of (\ref{1.1}), which indicates us to consider analyticity of solutions. Inspired by  the property of Gevrey multiplier observed by Oliver, Titi \cite{OT} where
\begin{equation}\label{R-E4}
\|\Lambda^{q}u\|_{L^2}^{2}\leq c(p,q)\tau^{p-2q}\|u\|_{L^2}\|\Lambda^{p}e^{\tau \Lambda}u\|_{L^2}\end{equation}
for $0\leq p \leq2q$ and $\tau>0$, we would develop a Gevrey method from Schonbek's theory for Kato's solution in $L^p$ critical framework. Our  strategy forwards Theorem \ref{thm2} contains the following three steps:
\begin{itemize}
\item Gevrey analyticity in the critical space
\vspace{4pt}
\item Evolution of Gevrey norm under the particular regularity $-\sigma$
\vspace{4pt}
\item Decay estimates
\end{itemize}

The basic idea to establish Gevrey analyticity in the critical space originated from works in \cite{BBT} and again, the effective velocity would be applied to deal with our coupling system and furnish the linear Gevrey estimates.

Main differences compared to previous work in \cite{BBT,S,L-cras} come from bounding nonlinear Gevrey products. Precisely, in case $1<p<6$, our main tool standardly derives from Lemari$\acute{e}$-Rieusset \cite{L-cras} and Bae, Biswas and Tadmor \cite{BBT}, where the Gevrey product was rewritten as a bilinear operator $\mathcal{B}(f,g)$ under Fourier symbol of $l_{1}$ norm, i.e.
\begin{eqnarray}m(\xi,\eta)=e^{\sqrt{t}(|\xi+\eta|_{1}-|\xi|_{1}-|\eta|_{1})}.\end{eqnarray}
In fact, $\mathcal{B}(f,g)$ could be regarded as finite linear combinations of identity operator, Hilbert transform, Poisson kernel which are $L^p$ ($1<p<\infty$) bounded.

However, the failure of boundedness for the Hilbert transform in $L^1$ space prevents us from the classical theory in the terminal case $p=1$. To deal with nonlinear estimates in the endpoint Lebesgue space, we introduce a technique modification with  bilinear operator $\mathcal{B}(f,g)$ equipped with symbol under Euclidean norm and Fourier localization functions
\begin{eqnarray}m(\xi,\eta)=e^{\sqrt{t}(c|\xi+\eta|-c_{1}|\xi|-c_{2}|\eta|)}\varphi(\xi)\varphi(\eta)\end{eqnarray}
with proper constants $c, c_{1},c_{2}$, whose pointwise estimates (for even high order derivatives) satisfy the Coifman-Meyer condition, see Camil, Wihelm \cite{CW}. Then the nonclassical bilinear multiplier theory concerns change for radius of analyticity by constants combined with the observation of equivalence of Gevrey multiplier with $l_{1}$ norm and Euclidean norm under localization enable us to have for some confirmed $d>1$
\begin{multline*}
\|e^{\sqrt {t}\Lambda_{1}}(u,\omega)\|_{E^1_{T}}+\|e^{d\sqrt {t}\Lambda_{1}}(u,\omega)\|_{E^p_{T}}
\lesssim\mathcal{X}_{0,1}+\big(\|e^{\sqrt {t}\Lambda_{1}}(u,\omega)\|_{E^1_{T}}+\|e^{d\sqrt {t}\Lambda_{1}}(u,\omega)\|_{E^p_{T}}\big)^{2}
\end{multline*}
and cover the Gevrey analyticity in the terminal case.

Another important issue is to search for the optimal upper bound for $\sigma$, which corresponds to the sharpness of decay rates. At this moment, we shall make fully use of nonlinear structure where divergence free condition allows us to rewrite nonlinearities into divergence forms. This observation leads to gain $|\xi|$ in those low-high-high frequency interactions and compensate the singularity in estimates. We state that $\|e^{\sqrt {t}\Lambda_{1}}(u,\omega)\|_{E^p_{T}}\ll1$ ensures one could prove the evolution of Gevrey norm under the regularity $-\sigma$ without any smallness assumption on initial data in $\dot B^{-\sigma}_{p,\infty}$.

Finally, based on the evolution of Gevrey norm for solutions under the regularity $-\sigma$, we obtain decay estimates by generalizing (\ref{R-E4}) into $L^p$ Besov framework (Lemma \ref{derivatives}) where we find solutions decay polynomially in the low frequencies while decay exponentially in the high frequencies.

The rest of this paper unfolds as follows. Section 3 is devoted to establishing the global well-posedness, that is Theorem \ref{thm1}. Section 4 is the part where we establish the large time behaviors. In Appendix, we give some functional boxes concern Besov space and classical Gevrey product estimates.

Throughout the paper, $C>0$ stands for a generic ``constant". For brevity, $f\lesssim g$ means that $f\leq Cg$. It will also be understood that $\|(f,g)\|_{X}=\|f\|_{X}+\|g\|_{X}$ for all $f,g\in X$. For $ 1\leq p\leq \infty$, we denote by $L^{p}=L^{p}(\mathbb{R}^{d})$ the usual Lebesgue space on $\mathbb{R}^{d}$ with the norm $\|\cdot\|_{L^{p}}$.

\section{Global well-posedness}

In this section, we shall prove Theorem \ref{thm1}. Different analysis would be imposed on high frequencies and low frequencies. Moreover, without loss of generality, we assume $\mu=\kappa=1$ while $\nu=\chi=\frac{1}{2}$.

For the high frequency part, since the coupling terms $\nabla\times\omega$ and $\nabla\times u$ are actually low degree terms, the following system will be consider:
\begin{equation}\label{3.1}
\left\{
\begin{array}{l}\partial_{t}u-\Delta u=\nabla\times\omega+f,\\ [1mm]
 \partial_{t}\omega-\Delta\omega-\nabla\mathrm{div}\omega=-2\omega+\nabla\times u +g,\\[1mm]
 \end{array} \right.
\end{equation}
 where $f=-\mathbf{P}[u\cdot\nabla u]$ and $g=-u\cdot\nabla\omega$.

The optimal smooth effect of the $Lam\acute{e}$ operator $\mathcal{L}u\triangleq \Delta u+\nabla \mathrm{div}u$ immediately yields for the high frequency cut-off $j>j_{0}$
\begin{eqnarray}\label{sum1}
\|u\|^{h}_{\tilde{L}^{\infty}_{T}(\dot{B}^{\frac{3}{p}-1}_{p,q})}+\|u\|^{h}_{\tilde{L}^{1}_{T}(\dot{B}^{\frac{3}{p}+1}_{p,q})}
\lesssim\|u_{0}\|^{h}_{\dot{B}^{\frac{3}{p}-1}_{p,q}}+\|\omega\|^{h}_{\tilde{L}^{1}_{T}(\dot{B}^{\frac{3}{p}}_{p,q})}+\|f\|^{h}_{\tilde{L}^{1}_{T}(\dot{B}^{\frac{3}{p}-1}_{p,q})};
\end{eqnarray}
\begin{multline}\label{sum2}
\,\,\,\,\|\omega\|^{h}_{\tilde{L}^{\infty}_{T}(\dot{B}^{\frac{3}{p}-1}_{p,q})}+\|\omega\|^{h}_{\tilde{L}^{1}_{T}(\dot{B}^{\frac{3}{p}+1}_{p,q})}
\lesssim\|\omega_{0}\|^{h}_{\dot{B}^{\frac{3}{p}-1}_{p,q}}+\|\omega\|^{h}_{\tilde{L}^{1}_{T}(\dot{B}^{\frac{3}{p}-1}_{p,q})}\\
+\|u\|^{h}_{\tilde{L}^{1}_{T}(\dot{B}^{\frac{3}{p}}_{p,q})}+\|g\|^{h}_{\tilde{L}^{1}_{T}(\dot{B}^{\frac{3}{p}-1}_{p,q})}.
\end{multline}
Now there holds
$$\|\omega\|^{h}_{\tilde{L}^{1}_{T}(\dot{B}^{\frac{3}{p}}_{p,q})}\lesssim 2^{-j_{0}}\|\omega\|^{h}_{\tilde{L}^{1}_{T}(\dot{B}^{\frac{3}{p}+1}_{p,q})};
\|u\|^{h}_{\tilde{L}^{1}_{T}(\dot{B}^{\frac{3}{p}}_{p,q})}\lesssim 2^{-j_{0}}\|u\|^{h}_{\tilde{L}^{1}_{T}(\dot{B}^{\frac{3}{p}+1}_{p,q})},$$
hence a large enough $j_{0}$ would ensure
\begin{eqnarray}\label{popoll}
\|(u,\omega)\|^{h}_{\tilde{L}^{\infty}_{T}(\dot{B}^{\frac{3}{p}-1}_{p,q})\cap\tilde{L}^{1}_{T}(\dot{B}^{\frac{3}{p}+1}_{p,q})}
\lesssim\|(u_{0},\omega_{0})\|^{h}_{\dot{B}^{\frac{3}{p}-1}_{p,q}}+\|(f,g)\|^{h}_{\tilde{L}^{1}_{T}(\dot{B}^{\frac{3}{p}-1}_{p,q})}.
\end{eqnarray}

To deal with the coupling system in the low frequencies, we use the method from Haspot \cite{ehome} which defined a new variable named effective velocity, that, somehow, allows to decouple the coupling the system. We underline that by the method of effective velocity, we may be able to reach the parabolic regularity for $u$ in the low frequencies and find the damping effect for $\omega$.

Let us recall the extended system (\ref{3.2}) under compressible/incompressible decomposition. Apparently Duhamel theory combining with localization multiplier implies $\mathcal{Q}\omega$ satisfies the following integral equation:
\begin{eqnarray}
\Delta_{j}\mathcal{Q}\omega=\Delta_{j}e^{\Delta t-2t}\mathcal{Q}\omega_{0}+\Delta_{j}\int^{t}_{0}e^{\Delta (t-s)-2(t-s)}\mathcal{Q} g(s)ds.
\end{eqnarray}
Now we claim that for multiplier $\Delta_{j}e^{\Delta t-2t}$, we have the following kernel estimate holds true for any distribution $f$ while $p\in[1,\infty]$:
$$\|\Delta_{j}e^{\Delta t-2t}f\|_{L^p}\lesssim e^{-(2^{2j}+2)t}\|\Delta_{j}f\|_{L^p}.$$
In fact, the above inequality is easily obtained by heat kernel estimates. Hence we have
\begin{eqnarray}\label{p.77}
\|\mathcal{Q}\omega\|^{\ell}_{\tilde L^{\infty}_{T}(\dot B^{\frac{3}{p}}_{p,q})\cap\tilde L^{1}_{T}(\dot B^{\frac{3}{p}}_{p,q})}\lesssim\|\mathcal{Q}\omega_{0}\|^{\ell}_{\dot B^{\frac{3}{p}}_{p,q}}+\|\mathcal{Q} g\|^{\ell}_{\tilde L^{1}_{T}(\dot B^{\frac{3}{p}}_{p,q})}.
\end{eqnarray}

As for the coupling system, we define a new variable $R=\nabla\times\mathcal{P}\omega+\frac{1}{2}\Delta u$. In terms of $R$, we are able to consider the following system which have been decoupled:
\begin{equation}\label{3.8}
\left\{
\begin{array}{l}\partial_{t}u-\frac{1}{2}\Delta u-R=f,\\ [1mm]
 \partial_{t}R+2R-\frac{3}{2}\Delta R=-\frac{1}{4}\Delta^{2}u+\frac{1}{2}\Delta f+\nabla\times g.\\[1mm]
 \end{array} \right.
\end{equation}
Similar calculations as $\mathcal{Q}\omega$ leads us to
\begin{multline}\label{4.70}
\|R\|^{\ell}_{\tilde L^{\infty}_{T}(\dot B^{\frac{3}{p}-1}_{p,q})\cap\tilde L^{1}_{T}(\dot B^{\frac{3}{p}-1}_{p,q})}\lesssim\|R_{0}\|^{\ell}_{\dot B^{\frac{3}{p}-1}_{p,q}}+\|\Delta^{2}u\|^{\ell}_{\tilde L^{1}_{T}(\dot B^{\frac{3}{p}-1}_{p,q})}+\|\Delta f+\nabla\times g\|^{\ell}_{\tilde L^{1}_{T}(\dot B^{\frac{3}{p}-1}_{p,q})}.
\end{multline}
Then we turn back to $u$ which satisfies
$$\partial_{t}u-\frac{1}{2}\Delta u=R+f.$$
Optimal smoothing effect for heat equations let us have
\begin{eqnarray}\label{4.7}
\|u\|^{\ell}_{\tilde L^{\infty}_{T}(\dot B^{\frac{3}{p}-1}_{p,q})\cap\tilde L^{1}_{T}(\dot B^{\frac{3}{p}+1}_{p,q})}\lesssim\|u_{0}\|^{\ell}_{\dot B^{\frac{3}{p}-1}_{p,q}}+\|R\|^{\ell}_{\tilde L^{1}_{T}(\dot B^{\frac{3}{p}-1}_{p,q})}+\|f\|^{\ell}_{\tilde L^{1}_{T}(\dot B^{\frac{3}{p}-1}_{p,q})}.
\end{eqnarray}
Therefore (\ref{4.70}) and (\ref{4.7}) combining with
$$\|\Delta^{2}u\|^{\ell}_{\tilde L^{1}_{T}(\dot B^{\frac{3}{p}-1}_{p,q})}\lesssim2^{2k_{0}}\|u\|^{\ell}_{\tilde L^{1}_{T}(\dot B^{\frac{3}{p}+1}_{p,q})},$$
we arrive at for frequency cut-off $k_{0}$ small enough such that
\begin{multline}\label{4.77}
\|u\|^{\ell}_{\tilde L^{\infty}_{T}(\dot B^{\frac{3}{p}-1}_{p,q})\cap\tilde L^{1}_{T}(\dot B^{\frac{3}{p}+1}_{p,q})}+\|R\|^{\ell}_{\tilde L^{\infty}_{T}(\dot B^{\frac{3}{p}-1}_{p,q})\cap\tilde L^{1}_{T}(\dot B^{\frac{3}{p}-1}_{p,q})}
\lesssim\|u_{0}\|^{\ell}_{\dot B^{\frac{3}{p}-1}_{p,q}}\\
+\|R_{0}\|^{\ell}_{\dot B^{\frac{3}{p}-1}_{p,q}}+\|f+\nabla\times g\|^{\ell}_{\tilde L^{1}_{T}(\dot B^{\frac{3}{p}-1}_{p,q})}.
\end{multline}
Naturally we find
\begin{multline}\label{4.8}
\|\nabla\times\mathcal{P}\omega\|^{\ell}_{\tilde L^{\infty}_{T}(\dot B^{\frac{3}{p}-1}_{p,q})\cap\tilde L^{1}_{T}(\dot B^{\frac{3}{p}-1}_{p,q})}\lesssim\|R+\Delta u\|^{\ell}_{\tilde L^{\infty}_{T}(\dot B^{\frac{3}{p}-1}_{p,q})\cap\tilde L^{1}_{T}(\dot B^{\frac{3}{p}-1}_{p,q})}\\
\lesssim\|u_{0}\|^{\ell}_{\dot B^{\frac{3}{p}-1}_{p,q}}+\|\nabla\times\mathcal{P}\omega_{0}\|^{\ell}_{\dot B^{\frac{3}{p}-1}_{p,q}}+\|f+\nabla\times g\|^{\ell}_{\tilde L^{1}_{T}(\dot B^{\frac{3}{p}-1}_{p,q})}.
\end{multline}
Because that $\mathrm{div}\mathcal{P}\omega=0$, we have
$$\|\nabla\times\mathcal{P}\omega\|^{\ell}_{\tilde L^{\infty}_{T}(\dot B^{\frac{3}{p}-1}_{p,q})\cap\tilde L^{1}_{T}(\dot B^{\frac{3}{p}-1}_{p,q})}\sim
\|\mathcal{P}\omega\|^{\ell}_{\tilde L^{\infty}_{T}(\dot B^{\frac{3}{p}}_{p,q})\cap\tilde L^{1}_{T}(\dot B^{\frac{3}{p}}_{p,q})}.$$
Hence (\ref{4.77}) and (\ref{4.8}) give:
\begin{multline}\label{lllppp}
\|u\|^{\ell}_{\tilde L^{\infty}_{T}(\dot B^{\frac{3}{p}-1}_{p,q})\cap\tilde L^{1}_{T}(\dot B^{\frac{3}{p}+1}_{p,q})}+\|\omega\|^{\ell}_{\tilde L^{\infty}_{T}(\dot B^{\frac{3}{p}}_{p,q})\cap\tilde L^{1}_{T}(\dot B^{\frac{3}{p}}_{p,q})}
\lesssim\|(u_{0},\Lambda \omega_{0})\|^{\ell}_{\dot B^{\frac{3}{p}-1}_{p,q}}+\|f+\nabla\times g\|^{\ell}_{\tilde L^{1}_{T}(\dot B^{\frac{3}{p}-1}_{p,q})}.
\end{multline}
Consequently, (\ref{p.77}), (\ref{lllppp}) let us finish the linear estimates with
\begin{eqnarray}
\|(u,\omega)\|_{E^p}\lesssim \mathcal{X}_{0,p}+\|(f,g)\|^{h}_{\tilde L^{1}_{T}(\dot B^{\frac{3}{p}-1}_{p,q})}+\|(f,\Lambda g)\|^{\ell}_{\tilde L^{1}_{T}(\dot B^{\frac{3}{p}-1}_{p,q})}.
\end{eqnarray}

Next we bound those nonlinearities, i.e. $\|(f,g)\|^{h}_{\tilde{L}^{1}_{T}(\dot{B}^{\frac{3}{p}-1}_{p,q})}$ and $\|(f,\Lambda g)\|^{\ell}_{\tilde L^{1}_{T}(\dot B^{\frac{3}{p}-1}_{p,q})}$. Actually the convection term $f$ shares the same calculations as in \cite{S} and we only treat with $g$. We have Bony decomposition
$$g=u\cdot\nabla\omega=T_{u}\nabla\omega+R(u,\nabla\omega)+T_{\nabla\omega}u.$$
Clearly, for the first one, there holds
\begin{eqnarray}
\|T_{u}\nabla\omega\|^{h}_{\tilde L^{1}_{T}(\dot B^{\frac{3}{p}-1}_{p,q})}
\lesssim\|u\|_{\tilde L^{\infty}_{T}(\dot B^{\frac{3}{p}-1}_{p,q})}\|\nabla\omega\|_{\tilde L^{1}_{T}(\dot B^{\frac{3}{p}}_{p,q})}.
\end{eqnarray}
Owing to the embedding in the low frequencies, the condition $s\in[1,2]$ ensures
$$\|\nabla\omega\|_{\tilde L^{1}_{T}(\dot B^{\frac{3}{p}}_{p,q})}\lesssim\|\omega\|^{h}_{\tilde L^{1}_{T}(\dot B^{\frac{3}{p}+1}_{p,q})}+
\|\omega\|^{\ell}_{\tilde L^{1}_{T}(\dot B^{\frac{3}{p}}_{p,q})}\lesssim\|\omega\|_{E^p_{T}},$$
and we deduce
\begin{eqnarray}
\|T_{u}\nabla\omega\|^{h}_{\tilde L^{1}_{T}(\dot B^{\frac{3}{p}-1}_{p,q})}\lesssim\|(u,\omega)\|^{2}_{E^p_{T}}.
\end{eqnarray}

As for remainder, since that $p$ satisfies $\frac{6}{p}-1>0$, there holds
\begin{eqnarray}
\|R(u,\nabla\omega)\|^{h}_{\tilde L^{1}_{T}(\dot B^{\frac{3}{p}-1}_{p,q})}
\lesssim\|u\|_{\tilde L^{\infty}_{T}(\dot B^{\frac{3}{p}-1}_{p,q})}\|\nabla\omega\|_{\tilde L^{1}_{T}(\dot B^{\frac{3}{p}}_{p,q})}\lesssim\|(u,\omega)\|^{2}_{E^p_{T}}.
\end{eqnarray}
As for the last, one, we infer that
\begin{multline}
\|T_{\nabla\omega}u\|^{h}_{\tilde L^{1}_{T}(\dot B^{\frac{3}{p}-1}_{p,q})}\lesssim\|T_{\nabla\omega}u\|^{h}_{\tilde L^{1}_{T}(\dot B^{\frac{3}{p}}_{p,q})}
\lesssim\|\nabla\omega\|_{\tilde L^{\infty}_{T}(\dot B^{\frac{3}{p}-1}_{p,q})}\|u\|_{\tilde L^{1}_{T}(\dot B^{\frac{3}{p}+1}_{p,q})}\lesssim\|(u,\omega)\|^{2}_{E^p_{T}}
\end{multline}
Therefore we get to
\begin{eqnarray}
\|g\|^{h}_{\tilde L^{1}_{T}(\dot B^{\frac{3}{p}-1}_{p,q})}\lesssim\|(u,\omega)\|^{2}_{E^p_{T}}.
\end{eqnarray}

For bounding $g$ in the low frequencies, there similarly holds
\begin{eqnarray*}
\|T_{u}\nabla\omega\|^{\ell}_{\tilde L^{1}_{T}(\dot B^{\frac{3}{p}}_{p,q})}\lesssim\|T_{u}\nabla\omega\|^{\ell}_{\tilde L^{1}_{T}(\dot B^{\frac{3}{p}-2}_{p,q})}\lesssim\|u\|_{\tilde L^{\infty}_{T}(\dot B^{\frac{3}{p}-1}_{p,q})}\|\nabla\omega\|_{\tilde L^{1}_{T}(\dot B^{\frac{3}{p}-1}_{p,q})};
\end{eqnarray*}
\begin{eqnarray*}
\|R(u,\nabla\omega^{\ell})\|^{\ell}_{\tilde L^{1}_{T}(\dot B^{\frac{3}{p}}_{p,q})}
\lesssim\|u\|_{\tilde L^{1}_{T}(\dot B^{\frac{3}{p}+1}_{p,q})}\|\nabla\omega^{\ell}\|_{\tilde L^{\infty}_{T}(\dot B^{\frac{3}{p}-1}_{p,q})};
\end{eqnarray*}
\begin{eqnarray*}
\|R(u,\nabla\omega^{h})\|^{\ell}_{\tilde L^{1}_{T}(\dot B^{\frac{3}{p}}_{p,q})}
\lesssim\|R(u,\nabla\omega^{h})\|^{\ell}_{\tilde L^{1}_{T}(\dot B^{\frac{3}{p}-1}_{p,q})}
\lesssim\|u\|_{\tilde L^{\infty}_{T}(\dot B^{\frac{3}{p}-1}_{p,q})}\|\nabla\omega^{h}\|_{\tilde L^{1}_{T}(\dot B^{\frac{3}{p}}_{p,q})};
\end{eqnarray*}
\begin{eqnarray*}
\|T_{\nabla\omega^{\ell}}u\|^{\ell}_{\tilde L^{1}_{T}(\dot B^{\frac{3}{p}}_{p,q})}
\lesssim\|\nabla\omega^{\ell}\|_{\tilde L^{\infty}_{T}(\dot B^{\frac{3}{p}-1}_{p,q})}\|u\|_{\tilde L^{1}_{T}(\dot B^{\frac{3}{p}+1}_{p,q})};
\end{eqnarray*}
\begin{eqnarray*}
\|T_{\nabla\omega^{h}}u\|^{\ell}_{\tilde L^{1}_{T}(\dot B^{\frac{3}{p}}_{p,q})}
\lesssim\|T_{\nabla\omega^{h}}u\|^{\ell}_{\tilde L^{1}_{T}(\dot B^{\frac{3}{p}-1}_{p,q})}
\lesssim\|\nabla\omega^{h}\|_{\tilde L^{\infty}_{T}(\dot B^{\frac{3}{p}-2}_{p,q})}\|u\|_{\tilde L^{1}_{T}(\dot B^{\frac{3}{p}+1}_{p,q})}.
\end{eqnarray*}
Hence above estimates let us arrive at
\begin{eqnarray}
\|g\|^{\ell}_{\tilde L^{1}_{T}(\dot B^{\frac{3}{p}}_{p,q})}\lesssim\|(u,\omega)\|^{2}_{E^p_{T}}.
\end{eqnarray}

Consequently, we finally arrive at
\begin{eqnarray}
\|(u,\omega)\|_{E^p_{T}}\lesssim \mathcal{X}_{0,p}+\|(u,\omega)\|^{2}_{E^p_{T}}.
\end{eqnarray}
Finally the global existence and uniqueness is obtained by standard fixed point theory and a continuity argument.

\section{Large time behaviors}

In this section, we present the proof of Theorem \ref{thm2}. The cornerstone of our approach is to establish Gevrey analyticity in the critical space and even lower regularity $-\sigma$. Analysis in general $L^p$ ($1<p<\infty$) would be somehow classical while attentions shall be paid on those endpoint frameworks where the method in \cite{BBT} couldn't cover the entire case since Gevrey norm equipped with Fourier symbol in $l^1$ norm brings singularity in the frequency space.
To overcome these difficulties,  some new idea which bases on the changes of radius of analyticity and an extended Coifman-Meyer argument, i.e. Corollary \ref{Euclidean}, would be applied to bound nonlinear interactions.

\subsection{Extended Coifman-Meyer theory and Gevrey multiplier}
\hspace*{\fill}

We would first introduce some lemmas concern the extended Coifman-Meyer argument which is essential to give Gevrey estimates in $L^1$ space.
The following modified version of general Coifman-Meyer theorem takes boundedness of bilinear operator under Fourier localization multiplier into considerations, proved in Guo, Nakanishi \cite{GN}:
\begin{lem}\label{mutiplier}
Let $\beta_{1},  \beta_{2}$ be the multiple index. Suppose it holds for $m(\xi,\eta)$ for $|\beta_{1}|,|\beta_{2}|\geq0$ that
$$\big|\partial^{\beta_{1}}_{\xi}\partial^{\beta_{2}}_{\eta}m(\xi,\eta)\big|\leq C_{\beta}|\xi|^{-|\beta_{1}|}|\eta|^{-|\beta_{2}|}.$$
Assume moreover that $m(\xi,\eta)$ is smooth of $\xi$, $\eta$ supported in $\xi\in \lambda_{1} \mathcal{C}$, $\eta\in \lambda_{2} \mathcal{C}$, then for $1\leq p, q, r\leq\infty$, then the associated bilinear multiplier operator $\mathcal{B}(f,g)$ maps $L^p\times L^q\rightarrow L^r$ satisfying
$$\|\mathcal{B}(f,g)\|_{L^r}\leq C\|f\|_{L^p}\|g\|_{L^q}$$
where constant $C$ is independent of $\lambda_{1}, \lambda_{2}$.
\end{lem}

Based on Lemma \ref{mutiplier}, we immediately give the following Corollary concerns Gevrey estimates for Euclidean norm which even covers the endpoint Lebesgue space in further analysis:
\begin{cor}\label{Euclidean}
Define $m(\xi,\eta)$ as:
\begin{eqnarray}\label{erika}m(\xi,\eta)=e^{\gamma(c|\xi+\eta|-c_{1}|\xi|-c_{2}|\eta|)}\end{eqnarray}
where constants $c,c_{1},c_{2}>0$ are confirmed. If $\xi\in \lambda_{1} \mathcal{C}$, $\eta\in \lambda_{2} \mathcal{C}$ for $\lambda_{1}\leq k\lambda_{2}$ with $k>0$, then
for any $1\leq r, p, q\leq\infty$, we have for some constant $C$
independent of $\gamma, \lambda_{1}, \lambda_{2}$
$$\|\mathcal{B}(f,g)\|_{L^r}\leq C\|f\|_{L^{p}}\|g\|_{L^{q}}.$$
\end{cor}

\begin{proof}
 To prove Corollary \ref{Euclidean}, it is enough to verify Coifman-Meyer conditions for $m(\xi,\eta)\varphi(\frac{\xi}{\lambda_{1}})\varphi(\frac{\eta}{\lambda_{2}})$. Indeed, since that $\lambda_{1}\leq k\lambda_{2}$, one could seek for a constant
 $c_{2}$ depending on $c,c_{1}$ to be large enough such that
 $$c|\xi+\eta|-c_{1}|\xi|-c_{2}|\eta|\leq -|\eta|\sim -\lambda_{2},$$
then we have
\begin{eqnarray*}\label{pointwise1}
|\partial_{\xi_{1}}m(\xi,\eta)|&=&\Big|\frac{\xi_{1}}{\lambda_{1}|\xi|}e^{\gamma(c|\xi+\eta|-c_{1}|\xi|-c_{2}|\eta|)}\varphi'(\frac{\xi}{\lambda_{1}})\varphi(\frac{\eta}{\lambda_{2}})\\
\nonumber&+&\gamma\big(\frac{c(\xi_{1}+\eta_{1})}{|\xi+\eta|}-\frac{c_{1}\xi_{1}}{|\xi|}\big)e^{\gamma(c|\xi+\eta|-c_{1}|\xi|-c_{2}|\eta|)}\varphi(\frac{\xi}{\lambda_{1}})\varphi(\frac{\eta}{\lambda_{2}})\Big|\\
\nonumber&\lesssim&\frac{1}{\lambda_{1}}+\frac{k\gamma\lambda_{2}}{\lambda_{1}}e^{-\gamma\lambda_{2}}\lesssim\frac{1}{|\xi|}
\end{eqnarray*}
and
\begin{eqnarray*}\label{pointwise2}
|\partial_{\eta_{1}}m(\xi,\eta)|&=&\Big|\frac{\xi_{1}}{\lambda_{2}|\xi|}e^{\gamma(c|\xi+\eta|-c_{1}|\xi|-c_{2}|\eta|)}\frac{1}{\lambda_{2}}\varphi(\frac{\xi}{\lambda_{1}})\varphi'(\frac{\eta}{\lambda_{2}})\\
\nonumber&+&\gamma\big(\frac{c(\xi_{1}+\eta_{1})}{|\xi+\eta|}-\frac{c_{2}\eta_{1}}{|\eta|}\big)e^{\gamma(|\xi+\eta|-c_{1}|\xi|-c_{2}|\eta|)}\varphi(\frac{\xi}{\lambda_{1}})\varphi(\frac{\eta}{\lambda_{2}})\Big|\\
\nonumber&\lesssim&\frac{1}{\lambda_{2}}+\frac{\gamma\lambda_{2}}{\lambda_{2}}e^{-\gamma\lambda_{2}}\lesssim\frac{1}{|\eta|}.
\end{eqnarray*}
Thus, Fa$\grave{a}$ di Bruno formula leads us to for $|\beta_{1}|,|\beta_{2}|\geq0$
$$\Big|\partial^{\beta_{1}}_{\xi}\partial^{\beta_{2}}_{\eta}m(\xi,\eta)\Big|\leq C_{\beta}|\xi|^{-|\beta_{1}|}|\eta|^{-|\beta_{2}|}.$$
Consequently, Lemma \ref{mutiplier} immediately yields Corollary \ref{Euclidean}.
\end{proof}

To link the Gevrey multiplier with different symbols, we would also like to introduce the following Lemma revealing the equivalence of Gevrey multiplier equipped with $l^1$ norm and Euclidean norm.
\begin{lem}\label{G1-G2}
Setting $1\leq s, p, r \leq\infty$. For  any $f\in \mathcal{S}'$, one could find positive constants $c_{1}<1, c_{2}>1$ depending on $d$ such that
\begin{eqnarray}
C_{1}\|e^{c_{1}\alpha\Lambda_{1}}f\|_{{\dot B}^{s}_{p,r}}\leq\|e^{\alpha\Lambda}f\|_{{\dot B}^{s}_{p,r}}\leq C_{2}\|e^{c_{2}\alpha\Lambda_{1}}f\|_{\dot{B}^{s}_{p,r}}.
\end{eqnarray}
where support of $\hat{f}(\xi)$ is in $\lambda\mathcal{C}$. $C_{1}, C_{2}$ are independent of $\lambda, \alpha$.
\end{lem}
\begin{proof}
We only prove the right side while the other case follows the similar steps. We write
$$e^{\alpha\Lambda}f=e^{\alpha\Lambda-c_{2}\alpha\Lambda_{1}}e^{c_{2}\alpha\Lambda_{1}}f,$$
by definition of Besov norm and technique of scaling, it is enough to prove the following inequality holds for some $c_{2}>0$:
 \begin{eqnarray}
\|K(x)\|_{L^1}\leq C.
\end{eqnarray}
with $\hat{K}(\xi)\triangleq e^{\alpha|\xi|-c_{2}\alpha|\xi|_{1}}\varphi(\xi)$.
We first have that
 \begin{eqnarray*}
\partial_{\xi_1}\hat{K}(\xi)&=&\Big[\big(\frac{\alpha\xi_{1}}{|\xi|}-c_{2}\alpha \mathrm{sgn}(\xi_{1})\big)\varphi(\xi)
+\varphi'(\xi)\Big]e^{\alpha|\xi|-c_{2}\alpha|\xi|_{1}},
\end{eqnarray*}

 \begin{eqnarray*}
\partial^{2}_{\xi_1}\hat{K}(\xi)&=&\Big[\big(\frac{\alpha}{|\xi|}-\frac{\alpha|\xi|-\alpha\xi^{2}_{1}}{|\xi|^{3}}-c_{2}\alpha \delta(\xi_{1})\big)\varphi(\xi)+\big(\frac{\alpha\xi_{1}}{|\xi|}-c_{2}\alpha \mathrm{sgn}(\xi_{1})\big)^{2}\varphi(\xi)\\
\nonumber&+&\big(\frac{\alpha\xi_{1}}{|\xi|}-c_{2}\alpha \mathrm{sgn}(\xi_{1})\big)\varphi'(\xi)+\varphi''(\xi)\Big]e^{\alpha|\xi|-c_{2}\alpha|\xi|_{1}}\\
\nonumber&=&\big(M(\xi)-c_{2}\alpha \delta(\xi_{1})\varphi(\xi)\big)e^{\alpha|\xi|-c_{2}\alpha|\xi|_{1}}
\end{eqnarray*}
where $\delta(x)$ is the delta function with $\langle \delta,f\rangle=f(0)$.

Utilizing that $e^{ix\cdot\xi}=\frac{1}{ix_{i}}\partial_{\xi_{i}}e^{ix\cdot\xi}$ and integral by parts, we observe that
 \begin{eqnarray*}
(1+x^2_{1})K(x)&=&\int_{\mathbb{R}^d}e^{ix\cdot\xi}(1+\partial^{2}_{\xi_{i}})\big(e^{\alpha|\xi|-c_{2}\alpha|\xi|_{1}}\varphi(\xi)\big)d\xi\\
\nonumber&=&\int_{\mathbb{R}^d}e^{ix\cdot\xi}(M(\xi)+\hat{K}(\xi))d\xi+c_{2}\alpha\langle \delta(\xi_{1}),e^{ix\cdot\xi}e^{\alpha|\xi|-c_{2}\alpha|\xi|_{1}}\varphi(\xi)\rangle.
\end{eqnarray*}

Now we introduce a family of functions $\varphi_{k}(\xi)$($1\leq k\leq d$) satisfying

$(i).\sum^{d}_{k=1}\varphi_{k}(\xi)\equiv 1 \,\,on\,\, \mathrm{supp} \varphi(\xi),$

$(ii).\mathrm{supp} \varphi_{k}(\xi)\subset \big\{\mathcal{C}(0,\frac{3}{4},\frac{8}{3}), |\xi_{k}|\geq c\big\}.$

Then it is clear that when $k=1$
 \begin{eqnarray}\label{CR7}c_{2}\alpha \langle\delta(\xi_{1}),e^{ix\cdot\xi}e^{\alpha|\xi|-c_{2}\alpha|\xi|_{1}}\varphi_{1}(\xi)\rangle
=c_{2}\alpha e^{ix\cdot\xi}e^{\alpha|\xi|-c_{2}\alpha|\xi|_{1}}\varphi_{1}(\xi)\big|_{\xi_{1}=0}=0.\end{eqnarray}
Also we have for $k\neq1$ that
 \begin{multline}\label{CR8}
c_{2}\alpha \langle\delta(\xi_{1}),e^{ix\cdot\xi}e^{\alpha|\xi|-c_{2}\alpha|\xi_{1}|}\varphi_{k}(\xi)\rangle=c_{2}\alpha e^{ix\cdot\xi}e^{\alpha|\xi|-c_{2}\alpha|\xi|_{1}}\varphi_{k}(\xi)\big|_{\xi_{1}=0}\\
=c_{2}\alpha e^{i\bar{x}\cdot\bar{\xi}}e^{\alpha|\bar{\xi}|-c_{2}\alpha|\bar{\xi}|_{1}}\varphi_{k}(\bar{\xi})
 \end{multline}
with $\bar{\xi}=(0,\xi_{2},...,\xi_{d})$.

By taking a proper $c_{2}$ large enough such that $(\alpha |\xi|+\alpha^{2} |\xi|^{2}) e^{\alpha|\xi|-c_{2}\alpha|\xi|_{1}}\leq e^{-\alpha|\xi|}$, we immediately have for the former term
 \begin{eqnarray}
\Big|\int_{\mathbb{R}^d}e^{ix\cdot\xi}(M(\xi)+\hat{K}(\xi))d\xi\Big|\leq C.
\end{eqnarray}
For the latter term, inspired by (\ref{CR7})-(\ref{CR8}) and the fact that $e^{\alpha|\bar{\xi}|-c_{2}\alpha|\bar{\xi}|_{1}}\leq e^{-\alpha|\bar{\xi}|}$ for some $c_{2}$, we have
 \begin{eqnarray*}
\Big|c_{2}\alpha\langle \delta(\xi_{1}),e^{ix\cdot\xi}e^{\alpha|\xi|-c_{2}\alpha|\xi|_{1}}\varphi(\xi)\rangle\Big|\leq c_{2}\alpha e^{-\alpha|\bar{\xi}|}\varphi_{k}(\bar{\xi})\leq C.
\end{eqnarray*}
which implicates
$$\big|(1+x^2_{1})K(x)\big|\leq\int_{\xi\in \mathcal{C}}(1+\frac{1}{|\xi|^2})d\xi\leq C.$$
Similar calculations to $x_{i}$ for $i=2,...,d$ lead us to
$$\prod^{d}_{i=1}\big|(1+x^2_{i})K(x)\big|\leq C$$
and proof of Lemma \ref{G1-G2} is finished.

\end{proof}

\subsection{Gevrey analyticity under critical regularity}
\hspace*{\fill}

In this subsection, we give Gevrey analyticity in the critical space. The method for case $1<p<\infty$ is standardly derived from \cite{BBT} while we state $p=1$ is more problematic especially facing those nonlinear Gevrey estimates. Our strategy is to apply the Coifman-Meyer theory established above. The change of radius of analyticity by constants would ensure the uniform Gevrey estimates even in endpoint Lebesgue space, For clarity, we write $(U,W)\triangleq (e^{\sqrt{t}\Lambda_{1}}u,e^{\sqrt{t}\Lambda_{1}}w)$ and further $(F,G)\triangleq (e^{\sqrt{t}\Lambda_{1}}f,e^{\sqrt{t}\Lambda_{1}}g)$.

Before we give the main goal in this section, we would like to present the following Lemma reveals the real analyticity of heat equations with general external force $f$, which reads as
\begin{eqnarray}\label{toy1.1}\left\{
  \begin{array}{ll}
    \partial_{t}v-\Delta v=f \,\,\,\,\,\,\,\,\mathrm{in}\,\,\,\,(0,T)\times \mathbb{R}^d, \\
    v|_{t=0}=v_{0}(x)\,\,\,\,\,\,\,\,\,\,\mathrm{on}\,\,\,\,\, \mathbb{R}^d.
  \end{array}
\right.\end{eqnarray}

Now we give Gevrey estimates for (\ref{toy1.1}):
\begin{lem}\label{lem4.21}
Let $s\in \mathbb{R}, 1\leq p\leq\infty $ and $1\leq \rho_2, r\leq \infty.$ Let $v$ be the solution of heat equation (\ref{toy1.1}). Then, we have for any $T>0$
\begin{eqnarray}\label{mmy}
\|e^{\sqrt{t}\Lambda_{1}}v\|_{\tilde{L}^{\rho_1}_{T}(\dot{B}^{s+\frac{2}{\rho_1}}_{p,r})}
\lesssim C\big(\|v_{0}\|_{\dot{B}^{s}_{p,r}}+\|e^{\sqrt{t}\Lambda_{1}}f\|_{\tilde{L}^{\rho_2}_{T}(\dot{B}^{s-2+\frac{2}{\rho_2}}_{p,r})}\big)
\end{eqnarray}
for all $\rho_1\in[\rho_2,\infty]$.
\end{lem}

 The proof of above Lemma has been given in \cite{BBT} where we just omit those details. We state our main goal in this subsection which is presented in the following Lemma:
\begin{lem}\label{lem4.0}
Assume $p,q\in[1,\infty]$. Let $p$ satisfies the condition in Theorem \ref{thm1} and $(u,\omega)$ be the solution to \eqref{1.1} constructed by Theorem \ref{thm1}. Then there exists a small constant $\eta$ that for any data $(u_{0},\omega_{0})$ satisfying
$\mathcal{X}_{0,p}\leq\eta$, the solution $(u,\omega)\in E^p$ fulfills $e^{\sqrt t\Lambda_{1}}(u,\omega)\in E^p$.
\end{lem}

\begin{proof}
In fact, obtaining uniform Gevrey estimates in the critical space for $1<p<6$ follows similar analysis in \cite{S} by applying Gevrey analysis under a different regularity. Let us just give a rough calculations for $\mathcal{Q}\omega$. In light of the third equation of (\ref{3.2}), $\mathcal{Q}\omega$ satisfies
\begin{eqnarray}
\mathcal{Q}\omega=e^{\Delta t-2t}\mathcal{Q}\omega_{0}+\Delta_{j}\int^{t}_{0}e^{\Delta (t-s)-2(t-s)}\mathcal{Q}g(s)ds.
\end{eqnarray}
Applying Gevrey multiplier and Fourier localization on the above equation yields
\begin{eqnarray}
\dot{\Delta}_{j}e^{\sqrt{t}\Lambda_{1}}\mathcal{Q}\omega=\dot{\Delta}_{j}e^{\sqrt{t}\Lambda_{1}}e^{\Delta t-2t}\mathcal{Q}\omega_{0}+\Delta_{j}e^{\sqrt{t}\Lambda_{1}}\int^{t}_{0}e^{\Delta (t-s)-2(t-s)}\mathcal{Q}g(s)ds.
\end{eqnarray}
According to \cite{BBT}, the following kernel estimates hold for some positive $c$ such that
$$\|\dot{\Delta}_{j}e^{\sqrt{t}\Lambda_{1}}e^{\Delta t-2t}f\|_{L^p}\lesssim e^{-(c2^{2j}+2)t}\|\dot{\Delta}_{j}f\|_{L^p};$$
$$\|\dot{\Delta}_{j}e^{(-\sqrt{t-\tau}+\sqrt{t}-\sqrt{\tau})\Lambda_{1}}f\|_{L^p}\lesssim\|\dot{\Delta}_{j}f\|_{L^p},$$
then we soon have
\begin{eqnarray*}
\|\mathcal{Q}W\|^{\ell}_{\tilde L^{\infty}_{T}(\dot B^{\frac{3}{p}}_{p,q})\cap\tilde L^{1}_{T}(\dot B^{\frac{3}{p}}_{p,q})}\lesssim\|\mathcal{Q}\omega_{0}\|^{\ell}_{\dot B^{\frac{3}{p}}_{p,q}}+\|\mathcal{Q}G\|^{\ell}_{\tilde L^{1}_{T}(\dot B^{\frac{3}{p}}_{p,q})};
\end{eqnarray*}
\begin{eqnarray*}
\|\mathcal{Q}W\|^{h}_{\tilde L^{\infty}_{T}(\dot B^{\frac{3}{p}-1}_{p,q})\cap\tilde L^{1}_{T}(\dot B^{\frac{3}{p}-s+2}_{p,q})}\lesssim\|\mathcal{Q}\omega_{0}\|^{h}_{\dot B^{\frac{3}{p}-1}_{p,q}}+\|\mathcal{Q}G\|^{h}_{\tilde L^{1}_{T}(\dot B^{\frac{3}{p}-1}_{p,q})}.
\end{eqnarray*}
By applying effective velocity and repeating above calculations on other variables, we get Gevrey linear estimates. On the other hand, nonlinear estimates is similar as those in proof of Theorem \ref{thm1} by Proposition \ref{propA1} and we omit the details. Therefore it holds
\begin{eqnarray}\label{Gevrey p}
\|(U,W)\|_{E^p_{T}}\lesssim \mathcal{X}_{0,p}+\|(U,W)\|^{2}_{E^p_{T}}.
\end{eqnarray}
Then the global-in-time existence and uniqueness of Gevrey analytic solutions in $E^p_{T}$ is obtained by a fixed point argument.

Next we turn to case $p=1$. We first claim the following inequality:
\begin{eqnarray}\label{Gevrey kp}
\|e^{d\sqrt {t}\Lambda_{1}}(u,\omega)\|_{E^p_{T}}\lesssim \mathcal{X}_{0,p}+\|e^{d\sqrt {t}\Lambda_{1}}(u,\omega)\|^{2}_{E^p_{T}}
\end{eqnarray}
where $d>1$ is confirmed later. Actually, above inequality is obtained by a small modification of (\ref{Gevrey p}) by changing the radius of analyticity by constants and we omit the details. Now we turn to $L^1$ boundedness where the linear analysis faces no problems and we have
\begin{eqnarray}
\|(U,W)\|_{E^1_{T}}\lesssim \mathcal{X}_{0,1}+\|F\|^{h}_{\tilde L^{1}_{T}(\dot B^{2}_{1,q})}
+\|G\|^{h}_{\tilde L^{1}_{T}(\dot B^{2}_{1,q})}+\|(F,\Lambda G)\|^{\ell}_{\tilde L^{1}_{T}(\dot B^{2}_{1,q})}.
\end{eqnarray}
Now we deal with nonlinearities. We shall first concentrate on the convection term $F$. Based on Bony decomposition and the structure of divergence free field, we have
\begin{equation}
\label{decom}\mathbf{P}[\mathrm{div}(u\otimes u)]=\mathbf{P}\big[\mathrm{div}[T_{u_{m}} u_{n}]_{mn}\big]+\mathbf{P}\big[\mathrm{div}[R(u_{m}, u_{n})]_{mn}\big]+\mathbf{P}\big[\mathrm{div}[T_{u_{n}} u_{m}]_{mn}\big].
\end{equation}
We begin with estimates on paraproducts where
\begin{multline*}
\|e^{\sqrt{t}\Lambda_{1}} \mathbf{P}\big[\mathrm{div}[T_{u_{m}} u_{n}]_{mn}\big]\|_{\tilde{L}^1_{T}(\dot{B}^{2}_{1,q})}\lesssim \|e^{\sqrt{t}\Lambda_{1}} T_{u} u\|_{\tilde{L}^1_{T}(\dot{B}^{3}_{1,q})}
\lesssim \|e^{c\sqrt{t}\Lambda} T_{u} u\|_{\tilde{L}^1_{T}(\dot{B}^{3}_{1,q})},
\end{multline*}
where we apply Lemma \ref{G1-G2} in the last inequality with constant $c>1$. By the definition of the paraproduct and of $\mathcal{B}_{t}(f,g)$ and in light of $\dot{S}_{j-1}U=\sum\limits_{j'\leq j-2}\dot\Delta_{j'}U$, it holds that
\begin{equation}\label{estimlocT}
e^{c\sqrt{t}\Lambda} T_{u} u=\sum_{j\in \mathbb{Z}}W_{j}\,\,\, \mbox{with}\,\,\,
W_{j}\triangleq \sum\limits_{j'\leq j-2}\mathcal{B}_t(\dot\Delta_{j'}e^{c_{1}\sqrt{t}\Lambda}u,\dot\Delta_j e^{c_{2}\sqrt{t}\Lambda}u)
\end{equation}
with the symbol
\begin{equation}\label{sign}
m(\xi,\eta)=e^{\sqrt{t}(c|\xi+\eta|-c_{1}|\xi|-c_{2}|\eta|)}\varphi(\frac{\xi}{2^{j'}})\varphi(\frac{\eta}{2^{j}}),
\end{equation}
where $c,c_{1},c_{2}$ settled as in Corollary \ref{Euclidean}.
Consequently, since that $j'\leq j-2$, by  applying Corollary \ref{Euclidean}, it holds for any $p>1$ such that
\begin{eqnarray*}
\|W_{j}\|_{L^{1}_{T}L^1}&\lesssim&\sum_{j'\leq j-2} \|\dot{\Delta}_{j'}e^{c_{1}\sqrt{t}\Lambda}u\|_{L^{\infty}_{T}L^{p'}}\|\dot{\Delta}_{j}e^{c_{2}\sqrt{t}\Lambda}u\|_{L^{1}_{T}L^{p}}\\
&\lesssim& \sum_{j'\leq j-2}2^{(\frac{3}{p}-2)j'}2^{2j'}\|\dot{\Delta}_{j'}e^{c_{1}\sqrt{t}\Lambda}u\|_{L^{\infty}_{T}L^1}\|\dot{\Delta}_{j}e^{c_{2}\sqrt{t}\Lambda}u\|_{L^{1}_{T}L^p}.
\end{eqnarray*}
Now if we take $p$ close enough to $1$ such that $\frac{3}{p}-2>0$, by H$\ddot{o}$lder inequality for series, we get
\begin{eqnarray*}
\|e^{c\sqrt {t}\Lambda} T_{u}u\|_{\tilde{L}^1_{T}(\dot{B}^{3}_{1,q})}
&\lesssim&\sum_{j'\leq j-2}2^{(\frac{3}{p}-2)(j'-j)}\|e^{c_{1}\sqrt{t}\Lambda}u\|_{\tilde{L}^\infty_{T}(\dot{B}^{2}_{1,q})} \|e^{c_{2}\sqrt{t}\Lambda}u\|_{\tilde{L}^1_{T}(\dot{B}^{\frac{3}{p}+1}_{p,q})}\\&\lesssim&\|e^{c_{1}\sqrt{t}\Lambda}u\|_{\tilde{L}^\infty_{T}(\dot{B}^{2}_{1,q})} \|e^{c_{2}\sqrt{t}\Lambda}u\|_{\tilde{L}^1_{T}(\dot{B}^{\frac{3}{p}+1}_{p,q})}.
\end{eqnarray*}

Now we further set $c_{1}$ small enough and $d$ large enough. Then with help of Lemma \ref{G1-G2}, one could take proper $c_{1}$ small enough while $d$ large enough such that
\begin{eqnarray}\label{susuqiu}\|e^{c_{1}\sqrt{t}\Lambda}u\|_{\tilde{L}^\infty_{T}(\dot{B}^{2}_{1,q})}\lesssim
\|U\|_{\tilde{L}^\infty_{T}(\dot{B}^{2}_{1,q})};\end{eqnarray}
\begin{eqnarray}
\|e^{c_{2}\sqrt{t}\Lambda}u\|_{\tilde{L}^1_{T} (\dot{B}^{\frac{3}{p}+1}_{p,q})}\lesssim
\|e^{d\sqrt{t}\Lambda_{1}}u\|_{\tilde{L}^1_{T} (\dot{B}^{\frac{3}{p}+1}_{p,q})}.\end{eqnarray}
Consequently, by the embedding,  we conclude with
\begin{eqnarray}\label{nb1}\|e^{\sqrt{t}\Lambda_{1}}\mathbf{P}\big[\mathrm{div}[T_{u_{m}} u_{n}]_{mn}\big]\|_{\tilde{L}^1_{T}(\dot{B}^{-\sigma}_{1,\infty})}\lesssim\|U\|_{\tilde{L}^\infty_{T}(\dot{B}^{2}_{1,q})}\|e^{d\sqrt{t}\Lambda_{1}}u\|_{\tilde{L}^1_{T} (\dot{B}^{\frac{3}{p}+1}_{p,q})}.\end{eqnarray}
The other paraproduct term enjoys the same estimates as above.

{\bf Estimates for remainders.}
We turn to estimate remainders. Again, by Lemma \ref{G1-G2} and the fact $\mathrm{div} u=0$, we shall only deal with  $\|e^{c\sqrt{t}\Lambda} R(u,  u)\|_{\tilde{L}^1_{T}(\dot{B}^{3}_{1,q})}$.
By the spectrum cut-off, one has
\begin{eqnarray}\label{remainder}e^{c\sqrt{t}\Lambda} \dot{\Delta}_{j}R(u,u)=\sum_{j\leq j'+2}\dot{\Delta}_{j}\mathcal{B}_t(\tilde{\dot{\Delta}}_{j'}e^{c_{1}\sqrt{t}\Lambda}u,\dot{\Delta}_{j'}e^{c_{2}\sqrt{t}\Lambda}u),\end{eqnarray}
where
\begin{equation}\label{sign2}
m(\xi,\eta)=e^{\sqrt{t}(c|\xi+\eta|-c_{1}|\xi|-c_{2}|\eta|)}\tilde{\varphi}(\frac{\xi}{2^{j}})\varphi(\frac{\eta}{2^{j}}).
\end{equation}
Then it holds by Corollary \ref{Euclidean} that we could find $c_{2}$ large enough
\begin{eqnarray*}
&&\|\dot{\Delta}_{j}\mathcal{B}_t(\tilde{\dot{\Delta}}_{j'}e^{c_{1}\sqrt{t}\Lambda}u,\dot{\Delta}_{j'}e^{c_{2}\sqrt{t}\Lambda}u)\|_{L^1_{T} L^1}\\
&\lesssim&\sum_{j'\geq j-2}\|\dot{\Delta}_{j'}e^{c_{1}\sqrt{t}\Lambda}u\|_{L^\infty_{T}L^2}\|\dot{\Delta}_{j'}e^{c_{2}\sqrt{t}\Lambda}u\|_{L^1_{T}L^2}\\
&\lesssim&\sum_{j'\geq j-2}2^{-j'}2^{2j'}\|\dot{\Delta}_{j'}e^{c_{1}\sqrt{t}\Lambda}u\|_{L^\infty_{T}L^1}2^{(\frac{3}{p}+1)j'}\|\dot{\Delta}_{j'}e^{c_{2}\sqrt{t}\Lambda}u\|_{L^1_{T}L^p}.
\end{eqnarray*}
Then one can immediately have
\begin{eqnarray}\label{nb2}
\|e^{c\sqrt{t}\Lambda} R(u,u)\|_{\tilde{L}^1_{T}(\dot{B}^{3}_{1,q})}
&\lesssim&\|e^{c_{1}\sqrt{t}\Lambda}u\|_{\tilde{L}^\infty_{T}(\dot{B}^{2}_{1,q})}\|e^{c_{2}\sqrt{t}\Lambda}u\|_{\tilde{L}^1_{T} (\dot{B}^{\frac{3}{p}+1}_{p,q})}\\
&\nonumber\lesssim&\|U\|_{\tilde{L}^\infty_{T}(\dot{B}^{2}_{1,q})}\|e^{d\sqrt{t}\Lambda_{1}}u\|_{\tilde{L}^1_{T} (\dot{B}^{\frac{3}{p}+1}_{p,q})}.
\end{eqnarray}
Thus, combining (\ref{nb1}) with (\ref{nb2}), we obtain
\begin{eqnarray*}
\|F\|_{\tilde{L}^1_{T}(\dot{B}^{2}_{1,q})}
\lesssim \|U\|_{\tilde{L}^\infty_{T}(\dot{B}^{2}_{1,q})}\|e^{d\sqrt{t}\Lambda_{1}}u\|_{\tilde{L}^1_{T} (\dot{B}^{\frac{3}{p}+1}_{p,q})}.
\end{eqnarray*}

As for $G$, we impose very similar calculations as convection terms and it holds
\begin{eqnarray*}
\| G\|^{h}_{\tilde{L}^1_{T}(\dot{B}^{3-s}_{1,q})}+\|\Lambda G\|^{\ell}_{\tilde{L}^1_{T}(\dot{B}^{2}_{1,q})}
\lesssim \|(U,W)\|_{E^1_{T}}\|e^{d\sqrt {t}\Lambda_{1}}(u,\omega)\|_{E^p_{T}}.
\end{eqnarray*}
At the end, we deduct the following inequality:
\begin{eqnarray}\label{hok}
\|(U,W)\|_{E^1_{T}}\lesssim \mathcal{X}_{0,1}+\|(U,W)\|_{E^1_{T}}\|e^{d\sqrt {t}\Lambda_{1}}(u,\omega)\|_{E^p_{T}}.
\end{eqnarray}

In light of Cauchy inequality, plugging (\ref{Gevrey kp}) with (\ref{hok}) allows us to arrive at
\begin{multline}\label{uniform 5}
\|(U,W)\|_{E^1_{T}}+\|e^{d\sqrt {t}\Lambda_{1}}(u,\omega)\|_{E^p_{T}}
\lesssim\mathcal{X}_{0,1}+\big(\|(U,W)\|_{E^1_{T}}+\|e^{d\sqrt {t}\Lambda_{1}}(u,\omega)\|_{E^p_{T}}\big)^{2}.
\end{multline}
where Sobolev embedding implies $\mathcal{X}_{0,p}\lesssim\mathcal{X}_{0,1}$.

 Consequently, (\ref{uniform 5}) combining with fixed point argument allows us to obtain Gevrey analyticity in $E^1$ and we finish the proof of Lemma \ref{lem4.0}.

\end{proof}

\subsection{Evolution of Gevrey norm under $-\sigma$}
\hspace*{\fill}

In this subsection, we shall prove the following estimates for any $T\geq0$:
\begin{eqnarray*}\label{qqrm}
\sup_{t\in[0,T]}\|e^{\sqrt t\Lambda_{1}}(u,\Lambda\omega)\|^{\ell}_{\dot{B}^{-\sigma}_{p,\infty}}\leq C
\end{eqnarray*}
where $C$ is a positive constant depends on $\mathcal{D}_{0,p}$.

For clarify, we denote $G^p_{T}$ whose norm is defined as
\begin{eqnarray*}
\|(u,\omega)\|_{G^p_{T}}=\|u\|^{\ell}_{{\tilde L}^{\infty}_{T}({{\dot B}^{-\sigma}}_{p,\infty})\cap{\tilde L}^{1}_{T}({{\dot B}^{-\sigma+2}}_{p,\infty})}
+\|\omega\|^{\ell}_{{\tilde L}^{\infty}_{T}({{\dot B}^{-\sigma+1}}_{p,\infty})\cap{\tilde L}^{1}_{T}({{\dot B}^{-\sigma+1}}_{p,\infty})}.
\end{eqnarray*}

 Our main goal in this subsection is to prove the following Lemma for the corresponding global solution constructed in Theorem \ref{thm1}:
\begin{lem}\label{tol}
Setting $p\in[1,\infty]$  satisfies the condition in Theorem \ref{thm1}. Assume $\sigma$ fulfills
 $$1-\frac{3}{p}<\sigma<\min\{1+\frac{3}{p},1+\frac{3}{p'}\}.$$
 Then if $p>1$, it holds for the corresponding Kato's solution in any $T\in[0,\infty]$ such that
\begin{eqnarray}\label{sigma1}
\|(U, W)\|_{G^p_{T}}\leq C\Big(\mathcal{D}_{0,p}+\|(U,W)\|_{E^p_{T}}(\|(U,W)\|_{E^p_{T}}+\|(U,W)\|_{G^p_{T}})\Big),
\end{eqnarray}
where constant $C$ is dependent of $p$.

Moreover, if $p=1$, then the Kato's solution satisfies
\begin{eqnarray}\label{sigma2}
\,\,\|(U, W)\|_{G^1_{T}}\leq C\Big(\mathcal{D}_{0,1}+\|e^{D\sqrt{t}\Lambda_{1}}(u,\omega)\|_{E^1_{T}}(\|(U,W)\|_{E^1_{T}}+\|(U,W)\|_{G^1_{T}})\Big),
\end{eqnarray}
where the constant $D>1$ and $C$ is dependent of $D$.
\end{lem}

\begin{proof}

{\it \underline{Case for $1<p<\infty$.}}

 Again we apply effective velocity and Lemma \ref{lem4.21} to get linear estimates where
\begin{eqnarray}\label{uun}
\|(U, W)\|_{G^p_{T}}\leq C\big(
\|(u_0,\Lambda \omega_0)\|^{\ell}_{\dot{B}^{-\sigma}_{p,\infty}}+\|(F,\Lambda G)\|^{\ell}_{\tilde L^{1}_{T}(\dot B^{-\sigma}_{p,\infty})}\big).
\end{eqnarray}
In order to deal with the nonlinear term, we claim that if $\sigma$ satisfies
 $$1-\frac{3}{p}<\sigma<\min\{1+\frac{3}{p},1+\frac{3}{p'}\},$$
 then the following inequality holds:
\begin{multline}\label{nonlinear1}
\|e^{\sqrt{t}\Lambda_{1}}(ab)\|_{\tilde{L}^{\rho}_{T}(\dot{B}^{1-\sigma}_{p,\infty})}\leq \|A\|_{\tilde{L}^{\rho_{1}}_{T}(\dot{B}^{1-\sigma}_{p,\infty})}\|B\|_{\tilde{L}^{\rho_{2}}_{T}(\dot{B}^{\frac{3}{p}}_{p,\infty})}
+\|B\|_{\tilde{L}^{\bar{\rho}_{1}}_{T}(\dot{B}^{\frac{3}{p}}_{p,\infty})}\|A\|_{\tilde{L}^{\bar{\rho}_{2}}_{T}(\dot{B}^{1-\sigma}_{p,\infty})}
\end{multline}
where $\frac{1}{\rho}=\frac{1}{\rho_{1}}+\frac{1}{\rho_{2}}=\frac{1}{\bar{\rho}_{1}}+\frac{1}{\bar{\rho}_{2}}$.

In fact, (\ref{nonlinear1}) is the direct result of Lemma \ref{propA1}. Since $f=\mathrm{div}(u\otimes u)$, (\ref{nonlinear1}) yields
\begin{eqnarray*}
\|F\|^{\ell}_{\tilde{L}^{1}_{T}(\dot{B}^{-\sigma}_{p,\infty})}\leq \|e^{\sqrt{t}\Lambda_{1}}(u\otimes u)\|_{\tilde L^{\infty}_{T}(\dot B^{-\sigma+1}_{p,\infty})}\lesssim
\| U\|_{\tilde L^{\infty}_{T}(\dot B^{-\sigma}_{p,\infty})}\|U\|_{\tilde L^{1}_{T}(\dot B^{\frac{3}{p}+1}_{p,\infty})}.
\end{eqnarray*}
Since that $1-\frac{3}{p}<\sigma$, there holds
\begin{eqnarray*}
\|U\|_{\tilde L^{\infty}_{T}(\dot B^{-\sigma}_{p,\infty})}\leq \|U\|^{\ell}_{\tilde L^{\infty}_{T}(\dot B^{-\sigma}_{p,\infty})}+\|U\|^{h}_{\tilde L^{\infty}_{T}(\dot B^{\frac{3}{p}-1}_{p,q})}
&\leq&\|U\|_{E^p_{T}}+\|U\|_{G^p_{T}}.
\end{eqnarray*}
Hence we arrive at
\begin{eqnarray}\label{koko2}
\|F\|^{\ell}_{\tilde{L}^{1}_{T}(\dot{B}^{-\sigma}_{p,\infty})}\lesssim\|U\|_{E^p_{T}}(\|U\|_{E^p_{T}}+\|U\|_{G^p_{T}}).
\end{eqnarray}

As for $G$, based on Bony decomposition, there holds
\begin{eqnarray*}
\|\Lambda e^{\sqrt{t}\Lambda_{1}}T_{u}\nabla\omega\|^{\ell}_{\tilde{L}^{1}_{T}(\dot{B}^{-\sigma}_{p,\infty})}
\leq\|e^{\sqrt{t}\Lambda_{1}}T_{u}\nabla\omega\|^{\ell}_{\tilde{L}^{1}_{T}(\dot{B}^{-\sigma-1}_{p,\infty})}\leq \|U\|_{\tilde L^{\infty}_{T}(\dot B^{-\sigma}_{p,\infty})}\|\nabla W\|_{\tilde L^{1}_{T}(\dot B^{\frac{3}{p}-1}_{p,\infty})}
\end{eqnarray*}
provided $1-\frac{3}{p}<\sigma$, while for the remainder, if $\sigma<\min\{1+\frac{3}{p},1+\frac{3}{p'}\}$, then:
\begin{eqnarray*}
\|\Lambda e^{\sqrt{t}\Lambda_{1}}R(u,\nabla\omega)\|^{\ell}_{\tilde{L}^{1}_{T}(\dot{B}^{-\sigma}_{p,\infty})}
\leq \|U\|_{\tilde L^{2}_{T}(\dot B^{-\sigma+1}_{p,\infty})}\| W\|_{\tilde L^{2}_{T}(\dot B^{\frac{3}{p}}_{p,\infty})}.
\end{eqnarray*}
Also, we have
\begin{eqnarray*}
\|\Lambda e^{\sqrt{t}\Lambda_{1}}T_{\nabla\omega^{\ell}}u\|^{\ell}_{\tilde{L}^{1}_{T}(\dot{B}^{-\sigma}_{p,\infty})}
\leq\|e^{\sqrt{t}\Lambda_{1}}T_{\nabla\omega^{\ell}}u\|^{\ell}_{\tilde{L}^{1}_{T}(\dot{B}^{-\sigma}_{p,\infty})}\leq \|\nabla W^{\ell}\|_{\tilde L^{2}_{T}(\dot B^{-\sigma}_{p,\infty})}\|U\|_{\tilde L^{2}_{T}(\dot B^{\frac{3}{p}}_{p,\infty})}
\end{eqnarray*}
and
\begin{eqnarray*}
\|\Lambda e^{\sqrt{t}\Lambda_{1}}T_{\nabla\omega^h}u\|^{\ell}_{\tilde{L}^{1}_{T}(\dot{B}^{-\sigma}_{p,\infty})}
\leq\|e^{\sqrt{t}\Lambda_{1}}T_{\nabla\omega^h}u\|^{\ell}_{\tilde{L}^{1}_{T}(\dot{B}^{-\sigma-1}_{p,\infty})}\leq \|\nabla W^{h}\|_{\tilde L^{\infty}_{T}(\dot B^{-\sigma-2}_{p,\infty})}\|U\|_{\tilde L^{1}_{T}(\dot B^{\frac{3}{p}+1}_{p,\infty})}.
\end{eqnarray*}
Then we conclude with
\begin{eqnarray}\label{koko3}
\|\Lambda G\|^{\ell}_{\tilde{L}^{1}_{T}(\dot{B}^{-\sigma}_{p,\infty})}\lesssim\|(U,W)\|_{E^p_{T}}(\|(U,W)\|_{E^p_{T}}+\|(U,W)\|_{G^p_{T}}).
\end{eqnarray}

Therefore combining (\ref{uun}), (\ref{koko2}) with (\ref{koko3}), (\ref{sigma1}) is proved.

{\it \underline{Case for $p=1$.}}

The case $p=1$ requires delicate calculations concerns Gevrey nonlinear estimates. In fact the main strategy is similar to Lemma \ref{lem4.0} where we would impose different radius of analyticity. We first claim that if $\mathcal{X}_{0,1}$ small enough, than we have
\begin{eqnarray*}
\|e^{D\sqrt{t}\Lambda_{1}}(u,\omega)\|_{E^1_{T}}\ll1,\\ [1mm]
 \end{eqnarray*}
where the constant $D>1$ would be confirmed later. Above uniform bound is easily got by slight modification on radius of analyticity from Lemma \ref{lem4.0}. Now linear calculations with $p=1$ implicate that
\begin{eqnarray}
\|(U, W)\|_{G^1_{T}}\leq C\big(
\|(u_0,\Lambda \omega_0)\|^{\ell}_{\dot{B}^{-\sigma}_{1,\infty}}+\|(F,\Lambda G)\|^{\ell}_{\tilde L^{1}_{T}(\dot B^{-\sigma}_{1,\infty})}\big).
\end{eqnarray}
We would like to take the convection term $f$ as an example while all other terms enjoy similar calculations. Firstly by (\ref{decom}), there holds
\begin{eqnarray*}
\|e^{\sqrt{t}\Lambda_{1}} \mathbf{P}\big[\mathrm{div}[T_{u_{m}} u_{n}]_{mn}\big]\|_{\tilde{L}^1_{T}(\dot{B}^{-\sigma}_{1,\infty})}
\lesssim \|e^{c\sqrt{t}\Lambda} T_{u} u\|_{\tilde{L}^1_{T}(\dot{B}^{-\sigma+1}_{1,\infty})}.
\end{eqnarray*}
By the definition of $W_{j}$ in (\ref{estimlocT}), Corollary \ref{Euclidean} yields one could select proper $c_{1},c_{2}$ such that
\begin{eqnarray*}
\|W_{j}\|_{L^{1}_{T}L^1}&\lesssim&\sum_{j'\leq j-2} \|\dot{\Delta}_{j'}e^{c_{1}\sqrt{t}\Lambda}u\|_{L^{\infty}_{T}L^{\infty}}\|\dot{\Delta}_{j}e^{c_{2}\sqrt{t}\Lambda}u\|_{L^{1}_{T}L^{1}}\\
&\lesssim& \sum_{j'\leq j-2}2^{(3+\sigma)j'}2^{-\sigma j'}\|\dot{\Delta}_{j'}e^{c_{1}\sqrt{t}\Lambda}u\|_{L^{\infty}_{T}L^1}2^{4j}\|\dot{\Delta}_{j}e^{c_{2}\sqrt{t}\Lambda}u\|_{L^{1}_{T}L^1}.
\end{eqnarray*}
Therefore taking advantage $\sigma>-2$, we get
\begin{eqnarray*}
\|e^{c\sqrt {t}\Lambda} T_{u}u\|_{\tilde{L}^1_{T}(\dot{B}^{-\sigma+1}_{1,\infty})}
&\lesssim&\|e^{c_{1}\sqrt{t}\Lambda}u\|_{\tilde{L}^\infty_{T}(\dot{B}^{\sigma}_{1,\infty})} \|e^{c_{2}\sqrt{t}\Lambda}u\|_{\tilde{L}^1_{T}(\dot{B}^{\frac{3}{p}+1}_{p,\infty})}.
\end{eqnarray*}

Hence it is enough to keep $D$ large enough such that
\begin{eqnarray}\label{lllll}\|e^{c_{1}\sqrt{t}\Lambda}u\|_{\tilde{L}^\infty_{T}(\dot{B}^{-\sigma}_{1,\infty})}\lesssim
\|U\|_{\tilde{L}^\infty_{T}(\dot{B}^{-\sigma}_{1,\infty})};\end{eqnarray}
\begin{eqnarray}\label{llllll}
\|e^{c_{2}\sqrt{t}\Lambda}u\|_{\tilde{L}^1_{T} (\dot{B}^{4}_{1,\infty})}\lesssim
\|e^{D\sqrt{t}\Lambda_{1}}u\|_{\tilde{L}^1_{T} (\dot{B}^{4}_{1,q})},\end{eqnarray}
which implies
\begin{eqnarray}\label{nb1}
\|e^{\sqrt{t}\Lambda_{1}}\mathbf{P}\big[\mathrm{div}[T_{u_{m}} u_{n}]_{mn}\big]\|_{\tilde{L}^1_{T}(\dot{B}^{-\sigma}_{1,\infty})}
\lesssim(\|U\|_{G^1_{T}}+\|U\|_{E^1_{T}})\|e^{D\sqrt{t}\Lambda_{1}}u\|_{E^1_{T}}.
\end{eqnarray}
The other paraproduct term enjoys the same estimates as above. For the remainder, there firstly holds
\begin{eqnarray*}
\|e^{\sqrt{t}\Lambda_{1}}\mathbf{P}\big[\mathrm{div}[R(u_{m}, u_{n})]_{mn}\big]\|_{\tilde{L}^1_{T}(\dot{B}^{-\sigma}_{1,\infty})}
\lesssim \|e^{\sqrt{t}\Lambda_{1}}R(u, u)\|_{\tilde{L}^1_{T}(\dot{B}^{-\sigma+1}_{1,\infty})},
\end{eqnarray*}
In light of (\ref{remainder}), and Corollary \ref{Euclidean}, we could find $c_{1}$ small enough and $c_{2}$ large enough
\begin{eqnarray*}
&&\|\dot{\Delta}_{j}\mathcal{B}_t(\tilde{\dot{\Delta}}_{j'}e^{c_{1}\sqrt{t}\Lambda}u,\dot{\Delta}_{j'}e^{c_{2}\sqrt{t}\Lambda}\nabla u)\|_{L^1_{T} L^1}\\
&\lesssim&\sum_{j'\geq j-2}\|\dot{\Delta}_{j'}e^{c_{1}\sqrt{t}\Lambda}u\|_{L^\infty_{T}L^2}\|\dot{\Delta}_{j'}e^{c_{2}\sqrt{t}\Lambda}u\|_{L^1_{T}L^2}\\
&\lesssim&\sum_{j'\geq j-2}2^{\sigma j'}2^{-\sigma j'}\|\dot{\Delta}_{j'}e^{c_{1}\sqrt{t}\Lambda}u\|_{L^\infty_{T}L^1}2^{3j'}\|\dot{\Delta}_{j'}e^{c_{2}\sqrt{t}\Lambda}\nabla u\|_{L^1_{T}L^1}.
\end{eqnarray*}
Since that $\sigma<0$, then one can immediately have
\begin{eqnarray*}
\|e^{c\sqrt{t}\Lambda} R(u,u)\|_{\tilde{L}^1_{T}(\dot{B}^{-\sigma}_{1,\infty})}
\lesssim\|e^{c_{1}\sqrt{t}\Lambda}u\|_{\tilde{L}^\infty_{T}(\dot{B}^{-\sigma}_{1,\infty})}\|e^{c_{2}\sqrt{t}\Lambda}u\|_{\tilde{L}^1_{T} (\dot{B}^{4}_{1,q})}.
\end{eqnarray*}
Keep in mind (\ref{lllll}) and (\ref{llllll}), we arrive at
\begin{eqnarray}
\|e^{\sqrt{t}\Lambda_{1}}\mathbf{P}\big[\mathrm{div}[R(u_{m}, u_{n})]_{mn}\big]\|_{\tilde{L}^1_{T}(\dot{B}^{-\sigma}_{1,\infty})}
\lesssim(\|U\|_{G^1_{T}}+\|U\|_{E^1_{T}})\|e^{D\sqrt{t}\Lambda_{1}}u\|_{E^1_{T}},
\end{eqnarray}
that is
$$\|F\|_{\tilde{L}^1_{T}(\dot{B}^{-\sigma}_{1,\infty})}
\lesssim(\|U\|_{G^1_{T}}+\|U\|_{E^1_{T}})\|e^{D\sqrt{t}\Lambda_{1}}u\|_{E^1_{T}}.$$
Estimating $G$ just follows similar steps and we arrive at (\ref{sigma1}) for $p=1$.
\end{proof}

By Theorem \ref{thm1}, for $p>1$, one is able to find a $\eta\leq\frac{1}{2\tilde{C}}$ that for $\mathcal{X}_{0,p}\leq\eta$, there holds
\begin{eqnarray*}
C\|(U,W)\|_{E^p_{T}}\leq \tilde{C}\mathcal{X}_{0,p}\leq\frac{1}{2},
\end{eqnarray*}
which lead to FOR $1<p<\infty$
\begin{eqnarray}\label{mnbv}
\|(U,\Lambda W)\|^{\ell}_{\tilde{L}^\infty_{T}(\dot{B}^{-\sigma}_{p,\infty})}\leq C\big(\mathcal{D}_{0,p}+1).
\end{eqnarray}
Similarly in case $p=1$, we have (\ref{mnbv}) by assuming $\mathcal{X}_{0,1}$ small enough such that $\|e^{D\sqrt{t}\Lambda_{1}}(u,\omega)\|_{E^1_{T}}\leq\frac{1}{2}$. Taking advantage of $\|f\|_{\tilde{L}^\infty_{T}(\dot{B}^{s}_{p,\infty})}\approx\|f\|_{L^\infty_{T}(\dot{B}^{s}_{p,\infty})}$, we conclude with (\ref{qqrm}).

\subsection{Decay estimates}
\hspace*{\fill}

In this section, we shall establish the decay rates for solutions constructed in Section 2, that is Theorem \ref{thm2}. Inspired by (\ref{R-E4}), Gevrey analyticity indicates us to transform derivatives into time decay and low order analytic estimates, which has been given by Lemma \ref{lem4.0} and we are able to obtain decay rates for $u$ as heat kernels while $\omega$ decays exponentially. To give a clear sight, we state the following lemma concerns properties of Gevrey multiplier:
\begin{lem}\label{derivatives}
Let $s\in \mathbb{R}$, $p\in[1,\infty]$. For any tempered distribution $f$, it holds for all $t>0$, $m>0$ that
\begin{eqnarray}\label{ahha1}\|\Lambda^m f\|^{\ell}_{{\dot B}^{s}_{p,1}}\leq C_{m} t^{-\frac{m}{2}}\|e^{\sqrt{t}\Lambda_{1}}f\|^{\ell}_{{\dot B}^{s}_{p,\infty}},\end{eqnarray}
where $C_{m}\triangleq\frac{1}{1-2^{-m}}+\big(8m\big)^{m}\frac{1}{1-e^{-\frac{1}{8}}}$ with $j_{0}$ represents frequency cut-off. Moreover, it holds for $t>0$, $m>0$ that
\begin{eqnarray}\label{ahha2}\|\Lambda^m f\|^{h}_{{\dot B}^{s}_{p,1}}\leq C_{m} t^{-\frac{m}{2}}e^{-a\sqrt{t}}\|e^{\sqrt{t}\Lambda_{1}}f\|^{h}_{{\dot B}^{s}_{p,\infty}},\end{eqnarray}
where $a>0$ is to be confirmed.
\end{lem}

\begin{proof}
By using Lemma \ref{lem5.1} and H$\ddot{o}$lder inequality, we have for frequency cut-off $j_{0}$ that
\begin{eqnarray}\label{Gd}
t^\frac{m}{2}\|\Lambda^m f\|^{\ell}_{{\dot B}^{s}_{p,1}}
&=&t^\frac{m}{2}\big\|2^{js}\|\Lambda^m e^{-\sqrt{t}\Lambda_{1}}\dot{\Delta}_{j}(e^{\sqrt{t}\Lambda_{1} }f)\|_{L^p}\big\|_{l^{1}_{j\leq j_{0}}}\\
\nonumber&\leq&C\big\|(\sqrt{t}2^{j})^m e^{-\frac{1}{2}\sqrt{t}2^{j}} 2^{js}\|\dot{\Delta}_{j}(e^{\sqrt{t}\Lambda_{1}}f)\|_{L^p}\big\|_{l^{1}_{j\leq j_{0}}}
\nonumber\\ \nonumber&\leq&C\big(\sum_{j\leq j_{0} }(\sqrt{t}2^{j})^m e^{-\frac{1}{2}\sqrt{t}2^{j}} \big)\|e^{\sqrt{t}\Lambda_{1}}f\|^{\ell}_{{\dot B}^{s}_{p,\infty}}.
\end{eqnarray}

We are left with bounding $\sum\limits_{j\leq j_{0} }(\sqrt{t}2^{j})^m e^{-\frac{1}{2}\sqrt{t}2^{j}} $ for given $t>0$. Indeed, one could always find an integer $p$ such that $4^{ p}<t\leq4^{p+1}$, at this moment, taking $l=j+p+1$, we have
$$\sum\limits_{j\leq j_{0}}(\sqrt{t} 2^j)^m e^{-\frac{1}{2}\sqrt{t} 2^{j}}\leq\sum\limits_{j\leq j_{0}}(2^{j+p+1})^m e^{-2^{j+p-1}}\leq\sum\limits_{l\in \mathbb{Z}}(2^{m})^l e^{-2^{l-2}}.$$

Now we claim that $\sum\limits_{l\in \mathbb{Z}}(2^{m})^l e^{-2^{l-2}}\leq C_{m}$ with $C_{m}\triangleq\frac{1}{1-2^{-m}}+\big(8m\big)^{m}\frac{1}{1-e^{-8}}$ for given $t$. Indeed,
\begin{eqnarray*}
\sum\limits_{l\in \mathbb{Z}}(2^{m})^l e^{-2^{l-2}}
&=&\sum\limits_{l\leq0}(2^{m})^l e^{-2^{l-2}}+\sum\limits_{l>0}(2^{m})^l e^{-2^{l-2}}=I_{1}+I_{2}.
\end{eqnarray*}
For $I_{1}$, one has
\begin{eqnarray*}
I_{1}\leq\sum\limits_{l\leq0}(2^m)^{l}=\frac{1}{1-2^{-m}}.
\end{eqnarray*}
For $I_{2}$, notice that for $m>0$
\begin{eqnarray}\label{poly-exp}
(2^{l})^m e^{-\frac{1}{8}2^{l}}
\leq (8m)^m e^{-m}.
\end{eqnarray}
Actually, define $f(x)=x^m e^{-\frac{1}{8}x}$, then $f'(x)=m x^{m-1} e^{-\frac{1}{8}x}-\frac{1}{8}x^m e^{-\frac{1}{8}x}$, thus $f'(8m)=0$, \\ so $f(x)\leq (8m)^m e^{-m}$,
which yields (\ref{poly-exp}). Consequently one has
\begin{eqnarray*}
I_{2}
\leq\big(8m\big)^{m}e^{-m}\sum\limits_{l>0}e^{-\frac{2^{l}}{8}}\leq\big(8m\big)^{m}e^{-m}\sum\limits_{l>0}e^{-\frac{l}{8}}=\big(8m\big)^{m}e^{-m}
\frac{e^{-\frac{1}{8}}}{1-e^{-\frac{1}{8}}}\leq\big(8m\big)^{m}\frac{1}{1-e^{-\frac{1}{8}}}.
\end{eqnarray*}
Thus we conclude with
\begin{eqnarray*}
t^\frac{m}{2}\|\Lambda^m f\|^{\ell}_{{\dot B}^{s}_{p,1}}
\leq C_{m}\|e^{\sqrt{t}\Lambda_{1}}f\|^{\ell}_{{\dot B}^{s}_{p,\infty}}.
\end{eqnarray*}

Furthermore, for the high frequencies, we arrive at
\begin{eqnarray}\label{Gd2}
t^\frac{m}{2}\|\Lambda^m f\|^{h}_{{\dot B}^{s}_{p,1}}
&=&t^\frac{m}{2}\big\|2^{js}\|\Lambda^m e^{-\sqrt{t}\Lambda}\dot{\Delta}_{j}(e^{\sqrt{t}\Lambda_{1}}f)\|_{L^p}\big\|_{l^{1}_{j\geq j_{0}}}\\
\nonumber&\leq&C\big\|( \sqrt{t} 2^j)^{m} e^{-\frac{1}{2}\sqrt{t} 2^{j}} 2^{js}\|\dot{\Delta}_{j}(e^{\sqrt{t}\Lambda_{1}}f)\|_{L^p}\big\|_{l^{1}_{j\geq j_{0}}}\\
 \nonumber&\leq&C\big(\sum_{j\geq j_{0} }( \sqrt{t} 2^{j})^{m} e^{-\frac{1}{4}\sqrt{t} 2^{j}}\big)e^{-\frac{1}{4}\sqrt{t} 2^{j}}\|e^{\sqrt{t}\Lambda_{1}}f\|^{h}_{{\dot B}^{s}_{p,\infty}}\\ \nonumber&\leq&C_{m}e^{-a\sqrt{t}}\|e^{\sqrt{t}\Lambda_{1}}f\|^{h}_{{\dot B}^{s}_{p,\infty}},
\end{eqnarray}
where we take $a=\frac{1}{4}2^{j_{0}}$. Hence, the proof of Lemma \ref{derivatives} is complete.
\end{proof}

Based on Lemma \ref{derivatives}, we now furnish the proof of Theorem \ref{thm2}. The following inequality holds true for $t>0$:
 \begin{eqnarray}\label{low-high01}
\|\Lambda^{l}(u,\omega)\|_{L^r}
\lesssim \|\Lambda^{l}(u,\omega)\|^{\ell}_{\dot{B}^{0}_{r,1}}+\|\Lambda^{l}(u,\omega)\|^{h}_{\dot{B}^{0}_{r,1}}.
\end{eqnarray}

{\it \underline{Polynomial decay for low frequencies.}}

Taking $m=l+\tilde{\sigma}, s=-\sigma$ in Lemma \ref{derivatives} yields the fact
solutions decay polynomially in the low frequencies:
\begin{eqnarray}\label{double}
\|\Lambda^{l}u\|^{\ell}_{\dot{B}^{0}_{r,1}}
\lesssim \|\Lambda^{l+\tilde{\sigma}}u\|^{\ell}_{\dot{B}^{-\sigma}_{p,1}}
\leq C_{l}t^{-\frac{l+\tilde{\sigma}}{2}}\|e^{\sqrt{t}\Lambda_{1}}u\|^{\ell}_{\dot{B}^{-\sigma}_{p,\infty}}
\end{eqnarray}
for $l>-\tilde{\sigma}$, where $\tilde{\sigma}\triangleq \sigma+\frac{3}{p}-\frac{3}{r}$.
Therefore, in light of (\ref{qqrm}), we conclude with
\begin{eqnarray}\label{double3}
\|\Lambda^{l}u\|^{\ell}_{\dot{B}^{0}_{r,1}}\leq C_{l}  t^{-\frac{\tilde{\sigma}}{2}-\frac{l}{2}}.
\end{eqnarray}
Similarly, for $\omega$, we have for $l>-\tilde{\sigma}+1$ such that
\begin{eqnarray}\label{double6}
\|\Lambda^{l}\omega\|^{\ell}_{\dot{B}^{0}_{r,1}}\leq C_{l}  t^{-\frac{\tilde{\sigma}}{2}+\frac{1}{2}-\frac{l}{2}}.
\end{eqnarray}

{\it \underline{Exponential decay for the high frequencies.}}

On the other hand, it is shown that the decay of the high frequencies of solutions is actually exponential in large time. Indeed, by taking advantage of embedding relationship in the high frequencies, i.e. $\dot{B}^{s_{1}}_{p,\infty}\hookrightarrow\dot{B}^{s_{2}}_{p,q}$ once $s_{1}>s_{2}$, we infer by Lemma \ref{derivatives} there exists a positive $c$ such that
\begin{eqnarray}\label{dh}
\|\Lambda^{l}u\|^{h}_{\dot{B}^{0}_{r,1}}
\lesssim\|\Lambda^{\gamma}u\|^{h}_{\dot{B}^{\frac{3}{p}-1}_{p,q}}\leq C_{\gamma} t^{-\frac{\gamma}{2}}e^{-c\sqrt{t}}\|e^{\sqrt{t}\Lambda_{1}}u\|^{h}_{\dot{B}^{\frac{3}{p}-1}_{p,q}}
\end{eqnarray}
provided  $\gamma> \max\{0,l+1-\frac{3}{r}\}$. Therefore, in light of Lemma \ref{lem4.0}, we obtain
\begin{eqnarray}\label{spacial}\|\Lambda^{l}u\|^{h}_{\dot{B}^{0}_{r,1}}\leq C_{\gamma} t^{-\frac{\gamma}{2}}e^{-c\sqrt{t}}.\end{eqnarray}
In terms of $\omega$, we also have for $\gamma> \max\{0,l+s-\frac{3}{r}\}$ such that
\begin{eqnarray}\label{spaciall}
\|\Lambda^{l}\omega\|^{h}_{\dot{B}^{0}_{r,1}}\leq C_{\gamma} t^{-\frac{\gamma}{2}}e^{-c\sqrt{t}}\|e^{\sqrt{t}\Lambda_{1}}\omega\|^{h}_{\dot{B}^{\frac{3}{p}-s}_{p,q}}
\leq C_{\gamma} t^{-\frac{\gamma}{2}}e^{-c\sqrt{t}}.
\end{eqnarray}
Therefor by , combining (\ref{double3})-(\ref{double6}) and (\ref{spacial})-(\ref{spaciall}) and taking $\gamma$ large enough, we finish the decay of $u$ in Theorem \ref{thm2}.\\

\noindent {\bf Conflicts of interest statement:}\ \

The author does not have any possible conflict of interest.

\section{Appendix}\setcounter{equation}{0}

\subsection{Littlewood-Paley theory and Besov space}
This section is devoted to give some preliminaries concerns Besov space and Gevrey multipliers. The Fourier transform $\widehat{f}=\mathcal{F}[f]$  of a function $f\in\mathcal{S}$ (the Schwartz class) is denoted by:
$$\widehat{f}(\xi)=\mathcal{F}[f](\xi):=\int_{\mathbb{R}^{N}}f(x)e^{-i\xi\cdot x}dx$$
for $ 1\leq p\leq \infty$, we denote by $L^{p}=L^{p}(\mathbb{R}^{N})$ the usual Lebesgue space on $\mathbb{R}^{N}$ with the norm $\|\cdot\|_{L^{p}}$.

Let us recall the Littlewood-Paley decomposition, the definitions of Besov spaces and Chemin-Lerner spaces, which are frequently used in the paper. Let $(\varphi,\chi)$ be a couple of smooth functions valued in $[0,1]$, such that $\varphi$ is
supported in the shell $\mathcal{C}(0,\frac{3}{4},\frac{8}{3})=\{\xi\in\mathbb{R}^{N}:\frac{3}{4}\leq|\xi|\leq\frac{8}{3}\}$, $\chi$ is supported in the ball $\mathcal{B}(0,\frac{4}{3})=\{\xi\in\mathbb{R}^{N}:|\xi|\leq\frac{4}{3}\}$ and
$$\forall\xi\in\mathbb{R}^{N},\quad \chi(\xi)+\sum_{q\in\mathbb{N}}\varphi(2^{-j}\xi)=1;$$
$$\forall\xi\in\mathbb{R}^{N}\backslash\{{0}\},\quad \sum_{q\in\mathbb{Z}}\varphi(2^{-j}\xi)=1.$$

For any tempered distribution $f\in\mathcal{S}'$, one can define localizations multiplier for $j\in \mathbb{Z}$ as follows:
\begin{eqnarray*}
&&\dot{\Delta}_{j}f:=\varphi(2^{-j}D)f=\mathcal{F}^{-1}(\varphi(2^{-j}\xi)\mathcal{F}f),
\end{eqnarray*}
\begin{eqnarray*}
&&\dot{S}_{j}f:=\chi(2^{-j}D)f=\mathcal{F}^{-1}(\chi(2^{-j}\xi)\mathcal{F}f).
\end{eqnarray*}

Denote by $\mathcal{S}'_{0}:=\mathcal{S'}/\mathcal{P}$ the tempered distributions modulo polynomials $\mathcal{P}$. It is well-known that Besov spaces can be
characterized by using the above spectral cut-off blocks.
The following Bernstein's lemma will be important throughout the paper.
\begin{lem}\label{lem2.1}
Let $1\leq p\leq q\leq\infty$. Then for any $\beta,\gamma\in(\mathbb{N}^+)^d$, we have a constant C independent of f,j such that
$$supp\widehat f\subseteq\{|\xi|\leq A_{0}2^j\} \Rightarrow \|\partial^\mu f\|_{L^q}\leq C_{\mu}2^{j|\mu|+dj(\frac{1}{p}-\frac{1}{q})}\|f\|_{L^p},$$
$$supp\widehat f\subseteq\{A_{1}2^j\leq|\xi|\leq A_{2}2^j\} \Rightarrow C_{\mu}2^{j|\mu|}\|f\|_{L^p}\leq \sup_{\mu}\|\partial^\mu f\|_{L^p}\leq C_{\mu}2^{j|\mu|}\|f\|_{L^p},$$
$$supp\widehat f\subseteq\{A_{1}2^j\leq|\xi|\leq A_{2}2^j\} \Rightarrow \| f\|_{L^q}\leq C2^{-dj(\frac{1}{p}-\frac{1}{q})}\|f\|_{L^p}.$$
\end{lem}

\begin{defn}\label{defn2.1}
For $s\in \mathbb{R}$ and $1\leq p,r\leq \infty$, the homogeneous Besov spaces $\dot{B}^s_{p,r}$ are defined by
$$\dot{B}^s_{p,r}:=\Big\{f\in \mathcal{S}'_{0}:\|f\|_{\dot{B}^s_{p,r}}<\infty  \Big\} ,$$
where
\begin{equation*}
\|f\|_{\dot{B}^s_{p,r}}:=\Big(\sum_{q\in\mathbb{Z}}(2^{qs}\|\dot{\Delta}_qf\|_{L^{p}})^{r}\Big)^{1/r}
\end{equation*}
with the usual convention if $r=\infty$.
\end{defn}

In what follows, we would like to present some analysis tools which will be used in the subsequent proof. The embedding properties will be used several times throughout the paper.
\begin{prop}\label{prop2.1}
Let  $s\in\mathbb{R}$. There holds
$$\dot{B}_{p_1,r}^{s}\hookrightarrow \dot{B}_{p_2,\tilde{r}}^{s-d(\frac{1}{p_{1}}-\frac{1}{p_{2}})}$$
 when $1\leq p_1\leq p_2\leq\infty$ and
$1\leq r \leq \tilde{r} \leq \infty$;
\end{prop}

Also, we introduce the following equivalent norm of homogeneous Besov space.
\begin{prop}\label{prop7.2}
Let $s\in \mathbb{R}$ and $1\leq p, r\leq \infty$. Let $(f_{j})_{j\in \mathbb{Z}} $ be a sequence of
$L^p$ functions such that $\sum\limits_{j\in\Z} f_j$ converges to some distribution $f$ in $\cS'_0$ and
$$\Bigl\|2^{js}\|f_j\|_{L^p(\R^d)}\Bigr\|_{\ell^r(\Z)}<\infty.$$
If $\mathrm{supp} \hat{f}_{j}\subset\mathcal{C}(0,2^jR_{1},2^jR_{2})$ for some $0<R_{1}<R_{2},$ then $f$ belongs to $\dot{B}^{s}_{p,r}$ and there exists a constant $C$ such that
 \begin{equation*}
\|f\|_{\dot{B}^{s}_{p,r}}\leq C\Bigl\|2^{js}\|f_j\|_{L^p(\R^d)}\Bigr\|_{\ell^r(\Z)}\cdotp
 \end{equation*}
 \end{prop}

 The interpolation relationship of Besov space would be a useful tool.
 \begin{prop} \label{prop7.3}
Setting $1\leq p,r_{1},r_{2}, r\leq \infty, \sigma_{1}\neq \sigma_{2}$ and $\theta \in (0,1)$. There holds that:
$$\|f\|_{\dot{B}_{p,r}^{\theta \sigma_{1}+(1-\theta )\sigma_{2}}}\lesssim \|f\| _{\dot{B}_{p,r_{1}}^{\sigma_{1}}}^{\theta} \|f\|_{\dot{B}_{p,r_2}^{\sigma_{2}}}^{1-\theta }$$
with $\frac{1}{r}=\frac{\theta}{r_{1}}+\frac{1-\theta}{r_{2}}$.
\end{prop}

Moreover, a class of mixed space-time Besov spaces are also used when studying the evolution PDEs, which was firstly
proposed by Chemin and Lerner in \cite{NS}.
\begin{defn}\label{defn2.2}
For $s\in \mathbb{R}, 1\leq r,\theta \leq \infty$, the homogeneous Chemin-Lerner spaces $\tilde{L}^\theta_{T}(\dot{B}^s_{p,r})$ are defined for any $T>0$ by
$$\tilde{L}^\theta_{T}(\dot{B}^s_{p,r}):=\Big\{f\in L^\theta(0,T;\mathcal{S}'_{0}) :\|f\|_{\tilde{L}^\theta_{T}(\dot{B}^s_{p,r})}<\infty  \Big\}, $$
where
$$\|f\|_{\tilde{L}^\theta_{T}(\dot{B}^s_{p,r})}:=\Big(\sum_{q\in \mathbb{Z}}(2^{js}\|\dot{\Delta}_jf\|_{L^\theta_{T}(L^{p})})^{r}\Big)^{1/r} $$
with the usual convention if  $r=\infty$.
\end{defn}

The Chemin-Lerner space $\widetilde{L}^{\theta}_{T}(X)$ with $X=\dot{B}^{s}_{p,r}$ may be linked with the standard spaces $L_{T}^{\theta}(X)$ by means of the Minkowski's inequality.
\begin{rem}\label{Rem2.1}
It holds that
$$\left\|f\right\|_{\widetilde{L}^{\theta}_{T}(X)}\leq\left\|f\right\|_{L^{\theta}_{T}(X)}\,\,\,
\mbox{if} \,\, \, r\geq\theta;\ \ \ \
\left\|f\right\|_{\widetilde{L}^{\theta}_{T}(X)}\geq\left\|f\right\|_{L^{\theta}_{T}(X)}\,\,\,
\mbox{if}\,\,\, r\leq\theta.
$$
\end{rem}

The product estimates in Besov space play an elementary role in our analysis, we first introduce the Bony decomposition where
$$ab=T_{a}b+R(a,b)+T_{b}a,$$
while
$$T_{a}b=\sum_{j'}\tilde{\dot{S}}_{j'-2}a\dot{\Delta}_{j'}b;\qquad
R(a,b)=\sum_{j'}\tilde{\dot{\Delta}}_{j'}a\dot{\Delta}_{j'}b.$$

The following paraproduct estimates and remainder estimates are classical:
\begin{lem}\label{product}
Let $1\leq p\leq\infty$, $1\leq q\leq\infty$. Then if $s_{1}<\frac{d}{p}$, then it holds
$$\|T_{a}b\|_{{\dot B}^{s_{1}+s_{2}-\frac{d}{p}}_{p,q}}\lesssim\|a\|_{{\dot B}^{s_{1}}_{p,\infty}}\|b\|_{{\dot B}^{s_{2}}_{p,q}}.$$
Moreover, if $s_{1}+s_{2}> d\max\{0, \frac{2}{p}-1\}$, then it holds
$$\|R(a,b)\|_{{\dot B}^{s_{1}+s_{2}-\frac{d}{p}}_{p,q}}\lesssim\|a\|_{{\dot B}^{s_{1}}_{p,\infty}}\|b\|_{{\dot B}^{s_{2}}_{p,q}}.$$
\end{lem}

The following lemmas give some descriptions for Gevrey multipliers
\begin{lem}\label{lem5.1} (\cite{CDX}). Let $(r,R)$ satisfy $0<r<R$. For any tempered distribution $u$ fulfilling $\mathrm{supp}\widehat u\subset \lambda C$,
there exists a constant $c>0$ such that for all $\zeta\in \mathbb{R}$ and $a>0$,  the following inequality holds for all $p\in [1,\infty]$:
$$\|\Lambda^\zeta e^{-a\Lambda_{1}} u\|_{L^p}\leq C\lambda^\zeta e^{-\frac{1}{2}a\Lambda_{1}}\|u\|_{L^p},$$
where $C(0,r,R)\triangleq \{\xi\in \mathbb{R}^d|r\leq |\xi|\leq R\}$ is the annulus. Moreover, the above inequality also holds true if we replace $\Lambda_{1}$ by $\Lambda$.
\end{lem}

\subsection{Gevrey estimates with $\Lambda_{1}$ symbol}
In this subsection, we mainly introduce some Gevrey product estimates and composite estimates which are frequently applied in our analysis.
Proving the Gevrey regularity of solutions in real analytic family will be based on continuity results for the family
$(\mathcal{B}_{t})_{t\geq 0}$ of bilinear operators, which are defined by
\begin{eqnarray*}
\mathcal{B}_{t}(f,g)(t,x)&=&e^{c_{0}\sqrt {t}\Lambda_1}(e^{-c_{0}\sqrt {t}\Lambda_1}fe^{-c_{0}\sqrt {t}\Lambda_1}g)(x)
\nonumber\\&=& \frac{1}{(2\pi)^{2d}}\int_{\mathbb{R}^d}\int_{\mathbb{R}^d}e^{ix\cdot(\xi+\eta)}e^{c_{0}\sqrt {t}(|\xi+\eta|_{1}-|\xi|_{1}-|\eta|_{1})}\widehat f(\xi)\widehat g(\eta)d\xi d\eta
\end{eqnarray*}
for some $c_0>0$. From \cite{BBT}, one can introduce the following operators acting on functions depending
on one real variable:
$$K_{1}f=\frac{1}{2\pi}\int_{0}^{\infty}e^{ix\xi}\widehat f(\xi)d\xi,$$
$$K_{-1}f=\frac{1}{2\pi}\int_{-\infty}^{0}e^{ix\xi}\widehat f(\xi)d\xi,$$
and define $L_{a,1}$ and $L_{a,-1}$ as follows:
$$L_{a,1}f=f\quad \mbox{and}\quad L_{a,-1}f=\frac{1}{2\pi}\int_{\mathbb{R}^d}e^{ix\xi}e^{-2a|\xi|}\widehat f(\xi)d\xi.$$
Set
$$Z_{t,\alpha,\beta}=K_{\beta_{1}}L_{c_{0}\sqrt {t},\alpha_{1}\beta_{1}}\otimes...\otimes K_{\beta_{d}}L_{c_{0}\sqrt {t},\alpha_{d}\beta_{d}}
\quad \mbox{and}\quad K_{\alpha}=K_{\alpha_{1}}\otimes ...\otimes K_{\alpha_{d}}$$
for $t\geq 0$, $\alpha=(\alpha_{1}, ... ,\alpha_{d})$ and $\beta=(\beta_{1}, ...,\beta_{d})\in \lbrace-1,1\rbrace^d$.
Then it follows that
$$\mathcal{B}_{t}(f,g)=\sum_{(\alpha,\beta,\gamma)\in({\{-1,1\}^{d})}^3}K_{\alpha}(Z_{t,\alpha,\beta}fZ_{t,\alpha,\gamma}g).$$
It is not difficult to see that $K_{\alpha},Z_{t,\alpha, \beta}$ are linear combinations of smooth homogeneous of
degree zero Fourier multipliers, which are bounded on $L^p$ for $1<p<\infty$. Consequently,
\begin{lem}\label{lem5.3}
For any $1<p,p_{1},p_{2}<\infty$ with $\frac{1}{p}=\frac{1}{p_{1}}+\frac{1}{p_{2}}$, we have for some constant $C$
independent of $t\geq0$,
\begin{eqnarray}\label{momok}\|\mathcal{B}_{t}(f,g)\|_{L^p}\leq C\|f\|_{L^{p_{1}}}\|g\|_{L^{p_{2}}}.\end{eqnarray}
\end{lem}

We first state the following product estimate based on Lemma \ref{lem5.3}:
\begin{prop}\cite{CDX}\label{propA1}
Let $1<p<\infty$, $1\leq q\leq\infty$. Denote $(A,B)=e^{c_{0}\sqrt{t}\Lambda_{1}}(a,b)$. Then if following $s_{1}<\frac{d}{p}$, inequality holds true for paraproducts
$$\|e^{c_{0}\sqrt{t}\Lambda_{1}}(T_{a}b)\|_{\dot{B}^{s_{1}+s_{2}-\frac{d}{p}}_{p,q}}\lesssim\|A\|_{\dot{B}^{s_{1}}_{p,\infty}}
\|B\|_{\dot{B}^{s_{2}}_{p,q}}.$$
Moreover, if $s_{1}+s_{2}>d\max\{0,\frac{2}{p}-1\}.$, we have following inequality for remainders holds true:
$$\|e^{c_{0}\sqrt{t}\Lambda_{1}}R(a,b)\|_{\dot{B}^{s_{1}+s_{2}-\frac{d}{p}}_{p,q}}\lesssim
\|A\|_{\dot{B}^{s_{1}}_{p,\infty}}\|B\|_{\dot{B}^{s_{2}}_{p,q}}.$$
\end{prop}

\begin{rem}
Above estimates could be easily extended to Chemin-Lerner space. In fact, we could deduce from (\ref{momok}) such that
\begin{eqnarray}\|\mathcal{B}_{t}(f,g)\|_{L^\rho_{T} L^p}\leq C\|f\|_{L^{\rho_{1}}_{T}L^{p_{1}}}\|g\|_{L^{\rho_{2}}_{T}L^{p_{2}}}\end{eqnarray}
where $\frac{1}{\rho}=\frac{1}{\rho_{2}}+\frac{1}{\rho_{2}}$. Therefore Proposition \ref{propA1} in Chemin-Lerner space is obtained.
\end{rem}

\end{document}